\documentclass[12pt]{article}

\usepackage[a4paper]{geometry}
\usepackage[utf8]{inputenc}
\usepackage[T1]{fontenc}
\usepackage[english]{babel}

\usepackage{fancyhdr}
\usepackage{amsmath}
\usepackage{amssymb}
\usepackage{amsthm}
\usepackage{xcolor}
\usepackage{graphicx}
\usepackage{pgfplots}
\usepackage{nameref}
\usepackage{enumerate}
\pgfplotsset{compat=1.11}
\usepackage{dsfont}
\usetikzlibrary{decorations.markings}
\usepackage{subfigure}
\usepackage{bm}
\usepackage{tikz}
\usepackage{subcaption} 
\usepackage{float}
\usepackage{tkz-euclide}
\usetikzlibrary{angles}
\usetikzlibrary{quotes}
\usetikzlibrary{patterns}
\usepackage{biblatex}
\usepackage{mathrsfs}

\usepackage[hidelinks]{hyperref}
\usepackage{hyperref}
\usepackage{csquotes}

\newcommand{\R}{\mathbb{R}}

\newcommand{\N}{\mathbb{N}}

\newcommand{\Z}{\mathbb{Z}}

\newcommand{\F}{\mathcal{F}}

\newcommand{\M}{\mathcal{M}}

\newcommand{\HM}{\mathcal{H}}

\newcommand{\LL}{\mathcal{L}}
\newcommand{\LM}{\mathscr{L}}
\newcommand{\T}{\mathcal{T}}
\newcommand{\PP}{\mathcal{P}}

\newcommand{\p}{\rho}
\newcommand{\RH}{Rankine-Hugoniot }
\newcommand{\1}{\mathds{1}}
\newtheorem{teo}{Theorem}[section]
\newtheorem{prop}[teo]{Proposition}

\newtheorem{lem}[teo]{Lemma}
\theoremstyle{definition}
\newtheorem{defi}[teo]{Definition}
\newtheorem{rmk}[teo]{Remark}

\title{Wave-front tracking for a quasi-linear scalar conservation law with hysteresis II: the case of Preisach}
\author{Fabio Bagagiolo\footnote{Department of Mathematics, University of Trento, Italy, fabio.bagagiolo@unitn.it}\ \ and Stefan Moreti\footnote{Department of Mathematics, University of Trento, Italy, stefan.moreti@unitn.it}}
\date{}

\begin{document}
\maketitle

\begin{abstract}
We consider the Cauchy problem for the quasi-linear scalar conservation law \[u_t+\F(u)_t+u_x=0,\] where $\F$ is a specific hysteresis operator. 
Hysteresis models a rate-independent memory relationship between the input $u$ and its output, giving a non-local feature to the equation.

In a previous work the authors studied the case when $\mathcal{F}$ is the Play operator. In the present article, we extend the analysis to the case of Preisach operator, which is probably the most versatile mathematical model to describe hysteresis in the applications, especially for the presence of some kind of internal variables. This fact has required a new analysis of the equation. 

Starting from the Riemann problem, we address the so-called wave-front tracking method for a solution to the Cauchy problem with bounded variation initial data. An entropy-like condition is also exploited for uniqueness. 
\end{abstract}


\section*{Introduction}\label{S0}
In the recent paper \cite{BFMS} the authors study the Cauchy problem for a conservation law with hysteresis nonlinearities as follows:
\begin{equation}\label{eq: intro1}
    \begin{cases}
        u_t+w_t+u_x=0 &\quad \text{on } \R \times [0,T),\\
        w(x,t)={\cal F}[u(x,\cdot),w^0(x)](t)&\quad \forall\ t\in[0,T),\ \text{a.e. } x\in\R\\
        u(x,0)=u_0(x) & \quad \text{in } \R, \\ w(x,0)=w_0(x) & \quad \text{in } \R.
    \end{cases}
\end{equation}

\noindent
where: $\cal F$ is a so-called hysteresis operator, namely representing a memory-dependent input-output relationship between the pair of scalar functions $(t\mapsto u(x,t),t\mapsto w(x,t))$, one for almost every $x$; $u_0$ is the initial datum for the solution $u$; $w^0$ is a suitable space-dependent function for the initial values of the output $w$. The memory dependence represented by $\cal F$ is rate-independent, which is the main characterization of the hysteresis phenomena. The main goal was, starting from the Riemann problem, to give an existence and uniqueness result for \eqref{eq: intro1} for general $BV$ (bounded variation) initial data, by applying the so called wave-front tracking method.

In \cite{BFMS} the case where $\cal F$ is the Play operator is considered. In the present paper we study the case where $\cal F$ is the Preisach operator which is more versatile and useful to analytically represent many hysteresis phenomena from the applications.

As already pointed out in \cite{BFMS}, the presence of the hysteretic (memory) term in \eqref{eq: intro1} gives a non-local feature to the PDE which requires to re-write the equation in the (more standard) form $u_t+f(u)u_x=0$, with $f=F'$, for a suitable definition of a flow function $F$ in order to encode the memory hysteretic effect (see \eqref{eq: flux}). The main novelties of this kind of problem are in the definition of the function $f$ which also in general depends on the output $w$ (which encodes the memory) and presents several non-standard singularities and discontinuities. 

When the hysteresis relationship is represented by the Play operator, as in \cite{BFMS}, the function $f$ above becomes of piecewise constant type which requires, for the Riemann problem,  a fine analysis and combination of shocks. In the present paper, with the Preisach hysteresis, $f$ turns out to be piecewise continuous and non-constant. This fact gives a big difference with respect to the previous case, because besides shocks one has the presence of several combinations of rarefaction waves, and this also reflects in the wave-front tracking procedure.

One of the main differences between Play and Preisach operators is that the second one takes account of the evolution of a family of so-called internal variables (which is a common feature of many real-world hysteresis phenomena). The evolution of such internal variables is encoded by the evolution of a switching function defined in a half plane (the Preisach plane) or equivalently by the evolution of a maximal antimonotone graph on that half plane viewed as the evolution in a suitable metric space. Hence, for the wave-front tracking procedure, in case of $BV$ data, a suitable definition and analysis of BV functions with values in a metric space must be considered.

As a consequence of the above discussion, the construction of a solution via the wave-front tracking method in this work involves several non-standard difficulties:
\begin{enumerate}[i)]
    \item unlike in the classical scalar case and in the Play operator case studied in \cite{BFMS}, interactions between rarefaction waves are present in the wave-front tracking;
    \item as in \cite{BFMS}, the function $u$ solving the Cauchy problem belongs only to BV. Hence, we need a relaxation of the hysteresis relationship, rewriting it as a suitable measure-dependent variational integral inequality (adapting the one in Visintin \cite{AVH});
    \item we are required to work in a suitable metric space of antimonotone graphs, which introduces additional difficulties in ensuring compactness and convergence of the wave-front tracking scheme and the relaxed hysteresis relationship.
\end{enumerate}
Our main results are the existence of a solution of \eqref{eq: intro1} and the uniqueness in the class of functions satisfying an entropy condition. 

A similar equation to \eqref{eq: intro1} but with a different hysteresis operator, namely the delayed Relay, was studied \cite{AVH1} by Visintin, where it is remarked that the results can be extended to the Preisach case. However, the approach followed there differs from that of \cite{BFMS} and from the one in the present article. Here, we adopt a more constructive approach focused on the study of characteristics and their interactions. 

Hysteresis is a phenomenon commonly observed in various natural and engineered systems, typically characterized by a lag or delay in the system’s response to changes in the input. Classical examples include the behavior of ferromagnetic materials, the stress-strain relation in plasto-elastic materials, and the behavior of thermostats. For comprehensive accounts of mathematical models for hysteresis and their use in connection with PDEs, we refer the reader to Krasnoselskii and Pokrovskii \cite{CORR1} and Visintin \cite{AVH}.

Hyperbolic and scalar conservation laws with hysteresis have been investigated in various applied contexts in works such as \cite{Simile}, \cite{ZHANG}, \cite{marchesin}, \cite{KOP1}, \cite{KOR1}, \cite{CF1}, \cite{MR}, \cite{CF2}, \cite{CF3}. Some of these works include an analysis of wave propagation and characteristics but no one of them uses a wave-front tracking approximation and its limit procedure to prove the existence of solutions. This indeed seems to be a novelty of our analysis, particularly for what concerns the passage to the limit in the hysteresis relationship. We refer to the Introduction of \cite{BFMS} for a brief discussion on the areas of interest and applications for some of those works. 

A wave-front tracking procedure is actually discussed in Fan \cite{F1} raising some questions about the convergence of the wave-front tracking limit concerning the hysteresis relation. Recently, Amadori, Bressan, and Shen in \cite{ADBA1}, \cite{ADBA2} study, using also wave-front tracking methods, a scalar conservation law in which the flux depends discontinuously on the spatial derivative $u_x$, with possible applications to  traffic flow with some kind of hysteresis effects.

For general basic theory on scalar conservation laws, we refer to the textbook by Evans \cite{EV}. For more specific results about the classical wave-front tracking method we refer to the books by Bressan \cite{AB3}, Holden and Risebro \cite{HH}, and the seminal paper by Dafermos \cite{DCP}.

The article is structured as follows. In Section \ref{S1} we introduce the Relay operator and its suitable extensions. Then in Section \ref{S2} instead, we present the Preisach operator $\F$ that appears in the PDE \eqref{eq: intro1}. In Section \ref{S3} we introduce a suitable formulation of \eqref{eq: intro1} and in Section \ref{S4} we study the corresponding Riemann problem. In Section \ref{S5} we prove existence of weak solutions for the case of $BV$ initial data, in particular we perform the limit in the wave-front tracking procedure. In Section \ref{S6} we give an entropy condition, showing that the solution constructed in the previous section is the only entropy solution. In the Appendix \ref{A1} we state some results about BV functions with values in metric spaces.


\numberwithin{equation}{section}


\section{The Relay Operator for hysteresis and its extensions}\label{S1}

Figure \ref{fig: Relay} represents the so-called input-output Relay hysteresis relationship between a time-dependent continuous scalar input $u$ and a time dependent scalar output $z.$ Let us fix a couple $(\p_1,\p_2)=:\p$, with $\p_1<\p_2$ and denote by \begin{equation}\label{eq: Lrho}
\LL_\p=\{ (u,z)\,|\,u > \p_1,\, z= +1\} \cup \{(u,z)\,|\, u < \p_2,\, z= -1\} 
\end{equation}
the so called hysteresis region, which is the region in $\R^2$ where the couple $(u(t),z(t))$ must belong. If at a certain time $t$, the couple $(u(t),z(t))$ is such that $z(t)=+1$, then certainly $u(t) > \p_1$, moreover, if the input changes in time, then the output will not change until $u$ possibly reaches the lower threshold $\p_1$. Once that threshold is reached then $z$ switches value from $+1$ to $-1$. If instead at some time $t$ the couple $(u(t),z(t))$ is such that $z(t)=-1,$ then certainly $u(t)<\rho_2$ and $z$ remains constant in time until $u$ possibly reaches the upper threshold $\p_2$. In that case $z$ at that time becomes $-1$.
\par 

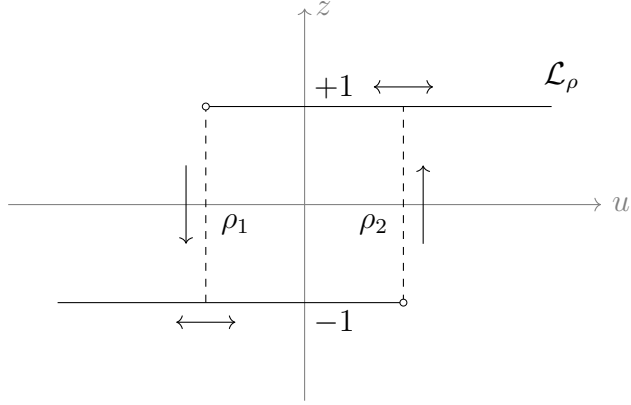
\begin{figure}
    \centering
    \begin{tikzpicture}[scale= 1.3]
        \draw[->, gray] (-3, 0) -- (3, 0) node[right] {$u$}; 
    \draw[->, gray] (0, -2) -- (0, 2) node[right] {$z$};
    \node at (0.7,-0.2) {$\rho_2$} ;
    \node at (-0.7,-0.2) {$\rho_1$} ;
    \node at (0.3,-1.2) {$-1$} ;
    \node at (0.3,1.2) {$+1$} ;
    \node at (2.6,1.3) {$\LL_\p$} ;
    \draw[-](-2.5,-1)--(1,-1);
    \draw[-](2.5,1)--(-1,1);
    \draw[-,dashed](-1,-1)--(-1,1);
    \draw[-,dashed](1,-1)--(1, 1);
    \draw[<->] (0.7,1.2)--(1.3,1.2);
    \draw[<->] (-0.7,-1.2)--(-1.3,-1.2);
    \draw[->] (1.2,-0.4)--(1.2,0.4);
    \draw[->] (-1.2,0.4)--(-1.2,-0.4);

    \filldraw [color=black, fill=white, ultra thin] (-1, 1) circle (1pt);
    \filldraw [color=black, fill=white, ultra thin] (1, -1) circle (1pt);
    \end{tikzpicture}
    \caption{Delayed Relay operator.}
    \label{fig: Relay}
\end{figure} 

Given a continuous input $u\in C^0([0,T])$ and an initial datum $z_0$, for the output, such that $(u(0),z_0)\in \LL_\p$, the previous heuristic argument can be made analytically rigorous defining then a unique output $z:[0,T]\to\mathbb{R}$, see \cite[Chapter~IV]{AVH}.

We will use the following notation $z=[h^\p(u,z_0)]$ to denote the output $z$ with initial state $z_0$ to the Relay operator with threshold $\p=(\p_1,\p_2)$ generated by input $u$. Sometimes we will highlight the dependence of $z$ from $\p$ by writing $z(\cdot\,; \p):[0,T] \to \{-1,1\}.$\par

Still referring to \cite{AVH}, for fixed $\p=(\p_1,\p_2)$, the Relay operator acts as follows
\[\begin{split}
h^\p: ~& \tilde{\cal L_\p}\to L^\infty(0,T)\cap BV([0,T])\cap PC(0,T)\\&(u(\cdot),z_0)\to [h^\p(u,z_0)](\cdot)=:z(\cdot) \end{split}\]
where $PC(0,T)$ is the space of piecewise constant functions and

\[
\tilde{\cal L_\p}=\left\{(u,z_0)\in C^0([0,T])\times\R\,\Big|\,(u(0),z_0)\in{\cal L_\p}\right\},
\]

\noindent
which satisfies the following properties (and in particular it is a hysteresis operator)
\begin{enumerate}[i)]
    \item \textit{Causality}: $\forall~(u_1,z_0), (u_2,z_0)$, such that $u_1=u_2$ in $[0,t]$ then 
    \[[h^\p(u_1,z_0)](t)=[h^\p(u_2,z_0)](t).\]
    \item \label{rateindep}\textit{Rate-independence}: $\forall~(u,z_0)$, $\forall ~t\in [0,T]$ if $s:~ [0,T]\to [0,T]$ is an increasing homeomorphism, then \begin{equation}\label{eq: rateindependence}[h^\p(u\circ s, z_0)](t)=[h^\p(u,z_0)](s(t)).\end{equation}
    \item \textit{Semigroup property}: $\forall~(u,z_0)$, $\forall ~t_1<t_2\in [0,T]$ setting $z(t_1):=[h^\p(u,z_0)](t_1)$ then we have that \begin{equation}\label{eq: semigroup}
    [h^\p(u,z_0)](t_2)=[h^\p(u(t_1+\cdot),z(t_1))](t_2-t_1).\end{equation}
\end{enumerate}

In the following, we will need to consider $h^\p$ as applied to $u$, with $u$ solution to a Riemann problem for a conservation law. Hence we have to extend the definition of $h^\p$ to piece-wise constant inputs with a finite number of jumps. Such an extension is the same as done in \cite{BFMS}: we consider a sequence of continuous functions $u_\varepsilon$ that fill the discontinuities of $u$ in a monotone way and such that $u_\varepsilon \to u$ pointwisely; then we define $z:=\lim_{\varepsilon \to 0} [h^\p(u_\varepsilon,z_0)]$ (see Figure \ref{fig: Relayapprox} for a specific example). In particular, such a construction of the output $z$, as function in $L^1(0,T)$, is independent of how we monotonically fill the jump of the input $u$ because of the rate-independence property \ref{rateindep}) above. 
\par

\begin{figure}
    \centering
    \begin{tikzpicture}[scale=0.8]
         \draw[->, gray, yshift=1cm] (-0.5, 0) -- (4, 0) node[right] {$t$}; 
        \draw[->, gray,yshift=1cm] (0, -0.5) -- (0, 2.5) node[right] {$u$};

        \draw[-,yshift=1cm](-0.1,1.5)--(0.1,1.5);
        \node[text = black, below] (r) at (-0.3,2.6) {\scriptsize{$\p_2$}};

        \draw[-,yshift=1cm](-0.1,-0.3)--(0.1,-0.3);
        \node[text = black, below] (r) at (-0.3,0.7) {\scriptsize{$\p_1$}};
        
        \draw[-,thick,yshift=1cm] (0,0)--(1.5,0);
        \draw[-,thick,yshift=1cm] (1.5,2)--(3,2);

        \draw[-,yshift=1cm](3,0.1)--(3,-0.1);
        \draw[-,dashed,yshift=1cm](1.5,0)--(1.5,2);

         \node[text = black, below] (r) at (1.5,1) {\scriptsize{$t^*$}};
        \node[text = black, below] (r) at (3,1) {\scriptsize{$T$}};

        \draw[->, gray, yshift=-3cm] (-0.5, 1) -- (4, 1) node[right] {$t$}; 
        \draw[->, gray,yshift=-3cm] (0, -0.5) -- (0, 2.5) node[right] {$z$};

        \draw[-,thick,yshift=-3cm] (0,0)--(1.5,0);
        \draw[-,thick,yshift=-3cm] (1.5,2)--(3,2);

        \draw[-,yshift=-3cm](3,1.1)--(3,0.9);
        \draw[-,dashed,yshift=-3cm](1.5,0)--(1.5,2);

         \node[text = black, below] (r) at (1.3,-2) {\scriptsize{$t^*$}};
        \node[text = black, below] (r) at (3,-2) {\scriptsize{$T$}};

        \draw[-](-0.1,-1)--(0.1,-1);
        \draw[-](-0.1,-3)--(0.1,-3);
        \node[text = black, below] (r) at (-0.3,-0.9) {\scriptsize{$+1$}};
        \node[text = black, below] (r) at (-0.3,-3) {\scriptsize{$-1$}};

        \draw[->, gray, yshift=1cm, xshift = 7cm] (-0.5, 0) -- (4, 0) node[right] {$t$}; 
        \draw[->, gray,yshift=1cm, xshift = 7cm] (0, -0.5) -- (0, 2.5) node[right] {$u_\varepsilon$};

        \draw[-,yshift=1cm, xshift = 7cm](-0.1,1.5)--(0.1,1.5);
        \node[text = black, below] (r) at (6.7,2.6) {\scriptsize{$\p_2$}};

        \draw[-,yshift=1cm, xshift = 7cm](-0.1,-0.3)--(0.1,-0.3);
        \node[text = black, below] (r) at (6.7,0.7) {\scriptsize{$\p_1$}};
        
        \draw[-,thick,yshift=1cm, xshift = 7cm] (0,0)--(1.3,0)--(1.7,2)--(3,2);
       
        \draw[-,yshift=1cm, xshift = 7cm](3,0.1)--(3,-0.1);
         \draw[-,yshift=1cm, xshift = 7cm](1.5,0.1)--(1.5,-0.1);

         \node[text = black, below] (r) at (8.5,1) {\scriptsize{$t^*$}};
        \node[text = black, below] (r) at (10,1) {\scriptsize{$T$}};

        \draw[->, gray, yshift=-3cm, xshift = 7cm] (-0.5, 1) -- (4, 1) node[right] {$t$}; 
        \draw[->, gray,yshift=-3cm, xshift = 7cm] (0, -0.5) -- (0, 2.5) node[right] {$z_\varepsilon$};

        \draw[-,thick,yshift=-3cm, xshift = 7cm] (0,0)--(1.6,0);
        \draw[-,thick,yshift=-3cm, xshift = 7cm] (1.6,2)--(3,2);

        \draw[-,yshift=-3cm, xshift = 7cm](3,1.1)--(3,0.9);
        \draw[-,yshift=-3cm, xshift = 7cm](1.5,1.1)--(1.5,0.9);
        \draw[-,dashed,yshift=-3cm, xshift = 7cm](1.6,0)--(1.6,2);

         \node[text = black, below] (r) at (8.3,-2) {\scriptsize{$t^*$}};
        \node[text = black, below] (r) at (10,-2) {\scriptsize{$T$}};

        \draw[-, xshift = 7cm](-0.1,-1)--(0.1,-1);
        \draw[-, xshift = 7cm](-0.1,-3)--(0.1,-3);
        \node[text = black, below] (r) at (6.7,-0.9) {\scriptsize{$+1$}};
        \node[text = black, below] (r) at (6.7,-3) {\scriptsize{$-1$}};
    
    \end{tikzpicture}
    \caption{From top-left clockwise: a piecewise constant input $u$; its continuous monotone approximation $u_\varepsilon$; the corresponding output $z_\varepsilon$; the limit output $z$; $u$ is such that after $t^*$, the threshold $\p_2$ is exceeded, hence in the limit $z$ jumps from $z_0=-1$ to $+1$ at time $t^*$.}
    \label{fig: Relayapprox}
\end{figure}
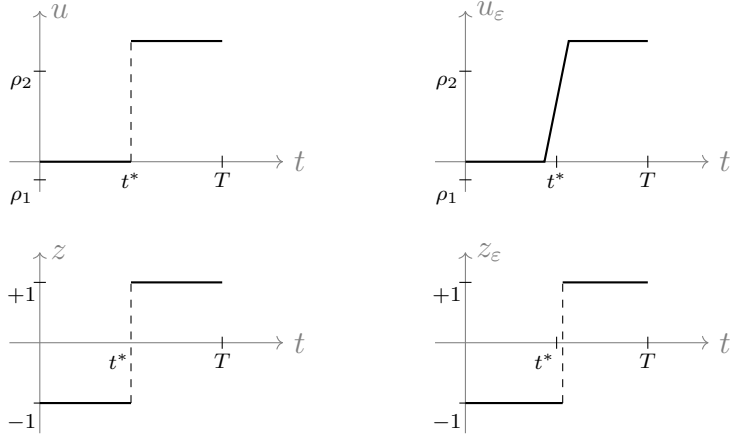

In the sequel of the paper we will also perform a limit of solutions of approximating Riemann problems, and hence we will need an extension of the Relay operator to even less regular inputs, as $BV$ functions. We have to weaken the Relay operator, closing in $\mathbb{R}^2$ the hysteresis region $\LL_\rho$ \eqref{eq: Lrho}, taking

\[
\bar{\LL}_\p=\{ (u,z)\,|\,u \geq \p_1,\, z= +1\} \cup \{(u,z)\,|\, u \leq \p_2,\, z= -1\}
\]
Given a continuous input $u:[0,T]\to \R$ and a piecewise constant function $z: [0,T]\to \{-1,1\}$, for fixed $\p=(\p_1,\p_2)$ and $z_0$ such that $(u(0),z_0)\in \bar{\LL}_\p$, we say that $z \in [\bar{h}^\p(u, z_0)]$ if and only if:
    \begin{enumerate}
        \item $z(0)=z_0$ and $(u(t),z(t))\in \bar{\LL}_\p$ $\forall\ t\in[0,T]$;
        \item if $u(t)\not=\p_1,\p_2$ then $z$ is constant in a neighborhood of $t$;
        \item if $u(t)=\p_2$ then $z$ is nondecreasing in a neighborhood of $t$;
        \item if $u(t)=\p_1$ then $z$ is nonincreasing in a neighborhood of $t$;
    \end{enumerate}
Note that the above definition may give more than one output for the same input (and indeed it is written as an inclusion $z \in [\bar{h}^\p(u, z_0)]$).
Also note that the points above define the closed Relay operator $\bar{h}^\p$ with the only difference from the previously defined $h^\p$ for the following greater freedom of the outputs in the threshold points (they may switch or not): 1) if at some time $t_1$, $z(t_1^-)=1$ and for $t_2\ge t_1$, $u(t)=\p_1$ in $[t_1,t_2]$ then for at most one time instant $t_3\in [t_1,t_2]$, from that moment on, $z$ may have changed value and this can be done only from $+1$ to $-1$, with any possible value $z(t_3)\in \{-1,+1\}$; 2) if at some time $t_1$, $z(t_1^-)=-1$ and for $t_2\ge t_1$, $u(t)=\p_2$ in $[t_1,t_2]$ then for at most one time instant $t_3\in [t_1,t_2]$, from that moment on, $z$ may have changed value and this can be done only from $-1$ to $+1$, with any possible value $z(t_3)\in \{-1,+1\}$. Finally note that $[h^\rho(u,z_0)]\in[{\bar h}^\rho(u.z_0)]$. For a more detailed view of this so-called completed Relay operator we refer to \cite[Section~6.1]{AVH}.

As above for piecewise constant functions $u,z$ we say that $z\in [\bar{h}^\p(u,z_0)]$ if there exists a sequence of continuous functions $u_\varepsilon$, filling in a monotone way the discontinuities of $u$, and a sequence $z_\varepsilon$ such that $z_\varepsilon \in [\bar{h}^\p(u_\varepsilon,z_0)]$ and $z=\lim_{\varepsilon\to0} z_\varepsilon$.\par Finally we can give the characterization via measure, which will be useful when working with $BV$ inputs $u$.
\begin{prop} \label{prop: Relayweak}
    Let us fix $z_0 \in \{-1,+1\}$ and a couple $\p=(\p_1,\p_2)$ with $\p_1<\p_2$, and consider a couple of function $u,z: [0,T] \to \R$ piecewise constant with a finite number of discontinuities. Moreover, denote by $\tilde{u}$ and $\tilde{z}$ the right-continuous representatives of $u$ and $z$ respectively and suppose $u(0)=\tilde{u}(0)$ and $z(0)=\tilde{z}(0)=z_0$, with $(u(0),z_0)\in\bar{\mathcal L}_\rho$. Then $(u,z)$ satisfies the Relay relationship $z\in[\bar{h}^{\p}(u,z_0)]$ if and only if \begin{enumerate}[1)]
        \item\label{prop1} for almost every $t\in [0,T]$, \begin{equation}\label{eq: weakhis}
            \int_{0}^{t} \tilde{u} ~d(Dz)\geq \int_{0}^{t} \p_2 ~d(D^+z)-\int_{0}^{t} \p_1 ~d(D^-z)=: \Psi_\p(z;(0,t));
        \end{equation}
        \item\label{prop2} for almost every $t\in [0,T]$, $|z|=1$ and \begin{equation} 
            \begin{cases}
                (z-1)(u-\p_2) \geq 0 \\
                (z+1)(u-\p_1)\geq 0,
            \end{cases}
        \end{equation}
        where $Dz$ denotes the measure associated to distributional derivative of $z$, and $D^+z$ and $D^-z$ respectively the positive and negative part of this measure.
    \end{enumerate}
\end{prop}

\begin{proof}
    ``$(\implies)$" Suppose $z\in[\bar{h}^\p(u,z_0)]$, then $\ref{prop2})$ follows immediately, since $(u,z)\in \bar{\mathcal{L}}_\rho$ for almost every $t$. Now we prove $\ref{prop1})$: if in the interval $[0,T]$, $z$ remains constant then $\ref{prop1})$ is trivially true. Now suppose that we have a single discontinuity for $z$ at time $t^*$, such that, e.g.,  $z$ jumps from $-1$ to $+1$. This means that $\tilde{u}(t^*)\geq \p_2$ and since \[ \int_{0}^{t} \tilde{u} ~d(Dz) = 2 \tilde{u}(t^*),\quad\int_{0}^{t} \p_2 ~d(D^+z) = 2 \p_2 \quad \text{and} \quad\int_{0}^{t} \p_1 ~d(D^-z) = 0,\] then $1)$ follows. Similarly we can deal with the case when $z$ jumps from $+1$ to $-1$ and proceeding recursively we can prove $\ref{prop1})$ for a finite number of jumps.

    ``$(\impliedby )$" Now assume instead $\ref{prop1})$ and $\ref{prop2})$. From $\ref{prop2})$ we have that whenever $z=1$ then $u\geq \p_1$ and whenever $z=-1$ then $u\leq \p_2$. So if $z $ is constant in an interval, this means that either $u$ has taken values in $(-\infty, \p_2]$ if $z\equiv -1$ or in $[ \p_1,+\infty)$ if $z\equiv +1$. Moreover, $\ref{prop2})$ implies that if $u>\p_2$ then $z=1$ and if $u<\p_1$ then $z=-1.$ Suppose now instead that $z$ has a single jump at time $t^*$, for example from $-1$ to $1$ then from $\ref{prop1})$ we get that $\tilde{u}(t^*)\geq \p_2$. Hence an approximation $u_\varepsilon$ as in Figure \ref{fig: Relayapprox} is easily constructed and the corresponding output $z_\varepsilon$ almost everywhere converge to $z$, showing that $z\in [\bar{h}^\p(u),z_0].$  By the same reasoning it can be $z \in [\bar{h}^\p(u,z_0)]$ for functions of a finite number of jumps. 
\end{proof}

The extension of hysteresis operators to piece-wise constant or BV inputs is just a part of the wider problem of extending them to non-regular inputs. The piece-wise continuous input approximation we performed can be seen as a special application of such extensions, for further references we then cite \cite{RF2}, \cite{RF1} and \cite{RV}.


\section{Preisach Operator}\label{S2}

Preisach operators are constructed by combining the action of a family of Relay operators with different thresholds, but with the same input working in parallel. We define first the Preisach plane as the half plane of all possible pairs of thresholds \begin{equation}
    \PP:=\{\p=(\p_1,\p_2)\in \R^2~|~ \p_1<\p_2\}.
\end{equation} Hence any point $\rho\in{\mathcal P}$ is identified with a Relay $\bar{h}^\rho$ and vice-versa.

Consider $\M$ the set of Borel measurable function from $\PP$ to $\{-1,1\}$ and denote by $z_0$ a generic element of $\M$, so for almost every $ \p=(\p_1,\p_2)\in \PP$ we have $z_0(\p) \in \{-1,1\}.$ The function $z_0$ represents for each $\p$ the initial state of the corresponding Relay $\bar{h}^\rho$ having threshold $\p$. We fix a parameter $a>0,$ and consider the triangle $\T$ on the plane $\PP$ with vertices $(-a,-a),(a,a)$ and $(-a,a)$. On $\T$ we then consider the two dimensional Lebesgue measure $\LM^2_{\llcorner \T}$.
\begin{defi}\label{def: POpe}
    The Preisach operator is defined as 
    \begin{equation}
    \label{eq:Preisach}
        \begin{split}
            \HM_{\LM^2_{\llcorner \T}} :~& \tilde{\LL} \subset C^0([0,T]) \times \M \to L^{\infty}(0,T) \\
            &(u,z_0)\to \int_{\PP} 
            z(t;\p)
            \,d \LM^2_{\llcorner \T} (\p)= \int_\T z(t;\p)\, d\p,
            \end{split}
        \end{equation}
        where $z(\cdot\,;\p)\in[\bar{h}^\p (u, z_0(x;\p))]\,$ and $\tilde{\LL}:=\{(u,z_0)\in C^0([0,T])\times \M \,|\, (u(0),z_0(\p)) \in \bar{\LL}_\p,\forall \p \in \PP\}.$ In the sequel, the condition $(u,z_0)\in\tilde{\LL}$ will be sometimes quoted saying that $u$ and $z_0$ are compatible.
\end{defi}

\begin{rmk}
    In Definition \ref{def: POpe} we wrote the explicit dependence of $\HM$ on $\LM^2_{\llcorner \T}$, since the Preisach operator can be constructed by choosing any finite (possibly signed) Borel measure $\mu$ over $\PP$. Nevertheless, in the following we will deal only with the operator $\HM_{\LM^2_{\llcorner \T}}$ so we drop the explicit dependence on $\LM^2_{\llcorner \T}$ and we write $w(\cdot):=[\HM(u,z_0)](\cdot)$ as the output of the operator. Moreover, as we are going to see (Remark \ref{rmk:scelta}), the choice of the Lebesgue measure yields the uniqueness of the output $w$ even if every Relay $\bar{h}^\rho$ still remains multi-valued. 
\end{rmk}

There is a useful interpretation of the evolution of the output $w$ by geometrical considerations on the Preisach plane. 

 Let us fix a function $u \in C^0([0,T])$ and $z_0=z_v \in \M$,
 which we call the virgin state (symmetric: never experienced switching evolution of the Relays), that is \begin{equation}
    z_v(\p)= \begin{cases}
        -1 \quad &\text{if} \quad \p_2 > - \p_1,\\
        +1 \quad &\text{if} \quad \p_2 < - \p_1.
    \end{cases}
\end{equation} Notice that the set $\{\p_2=-\p_1\}$ has Lebesgue measure $0$ so it is not necessary to specify the values of $z_0$ on that set, moreover, since the measure $\LM^2_{\llcorner \T}$ is only supported on $\T$, we consider $z_0$ restricted only to $\T$. We set then 
\begin{equation}
    \label{eq:csi}
\xi_z(t):=\partial \{ \p \in \PP ~|~ z(t;\p) =-1 \}\cap \partial \{ \p \in \PP ~|~ z(t;\p) =+1\}
\end{equation}
and note that by \eqref{eq:Preisach} it is
\begin{equation}
    \label{eq:w}
w(t)=\LM^2_{\llcorner \T}\left(\{ \p \in \PP ~|~ z(t;\p) =+1\}\right)-\LM^2_{\llcorner \T}\left(\{ \p \in \PP ~|~ z(t;\p) =-1 \}\right)
\end{equation}

Suppose that $u(0)=0$ and that it is non decreasing in the time interval $[0,t_1]$. By the definition of the Relay operator we conclude that  \begin{equation}\label{eq: zpp} z(t_1;\p)=
    \begin{cases}
        -1 \quad & \text{if} \quad u(t_1)< \p_1,\\
        +1 \quad & \text{if} \quad u(t_1) > \p_2,\\
        z_v \quad & \text{if} \quad \p_1<u(t_1)<\p_2,
    \end{cases}
\end{equation}
because the only possible switches in the time interval $[0,t_1]$ are from $-1$ to $1$ for those Relays whose upper threshold $\rho_2$ is bypassed by the non decreasing input $u$.
So $\xi_z(t_1)$ is the union of a horizontal segment right-anchored at $\rho=(u(t_1),u(t_1))$ and part of the original diagonal $\{\p_1=-\p_2\}$. Then suppose that $u$ is non-increasing in the time interval $[t_1,t_2]$. Then the only possible switches in the time interval $[t_1,t_2]$ are from $1$ to $-1$ for those Relays whose lower threshold $\rho_1$ is bypassed by the non-increasing input $u$. In this case $\xi_z(t_2)$ is the union of a vertical segment bottom-anchored at $\rho=(u(t_2),u(t_2))$, the remaining part of the horizontal segment here above and of the part of the original diagonal of the virgin state. For a piecewise monotone input $u$, we can repeat the alternating considerations above and see that, when $u$ increases, $\xi_z$ moves up generating horizontal segments, when $u$ decreases, $\xi_z$ goes to the left generating vertical segments (see Figure \ref{fig: prei2}). Notice that the intersection between $\xi_z(t)$ and the diagonal $\p_1=\p_2$ is always at $(u(t), u(t)).$ 

\begin{rmk}\label{rmk:scelta}Note that in all cases here above, in view of \eqref{eq:w}, the state of the Relays with threshold $\rho\in\xi_z$ is irrelevant because $\xi_z$ has null Lebesgue measure. Since the multivalued feature of the closed delayed Relay $\bar{h}^\rho$ may play a role only when $\rho\in\xi_z(t)$, we conclude that \eqref{eq:Preisach} is not more multivalued and that it uniquely defines the output $w$ of the Preisach operator which, by \eqref{eq:csi} and \eqref{eq:w} is uniquely determined by the evolution of the graph $t\mapsto\xi(t)$ on the Preisach plane. In other words, regardless of the choice of the selection $z\in[\bar{h}^\p (u, z_0(\p))]$, formula \eqref{eq:Preisach} always gives the same output $w$. Hence in the sequel, when dealing with the Preisach operator, we will always assume that in \eqref{eq:Preisach} we take $z(\cdot,\rho)=[h^\rho(u,z_0(\rho))](\cdot)$.

Moreover, it can be proven that for any continuous input $u$ (not only piecewise monotone), starting from $z_v$, the corresponding $\xi_z$ will be an antimonotone graph, union of part of the diagonal $\p_1=-\p_2$ with an either finite or countable family of alternating vertical and horizontal segments, possibly accumulating at $(u(t),u(t))$. The inverse is also true, for any $\xi_z$ of that form we can find a function $u$ that generates it starting from $\xi_v$ (see \cite[Theorem~2.2 Chapter~IV]{AVH}). 
\end{rmk}

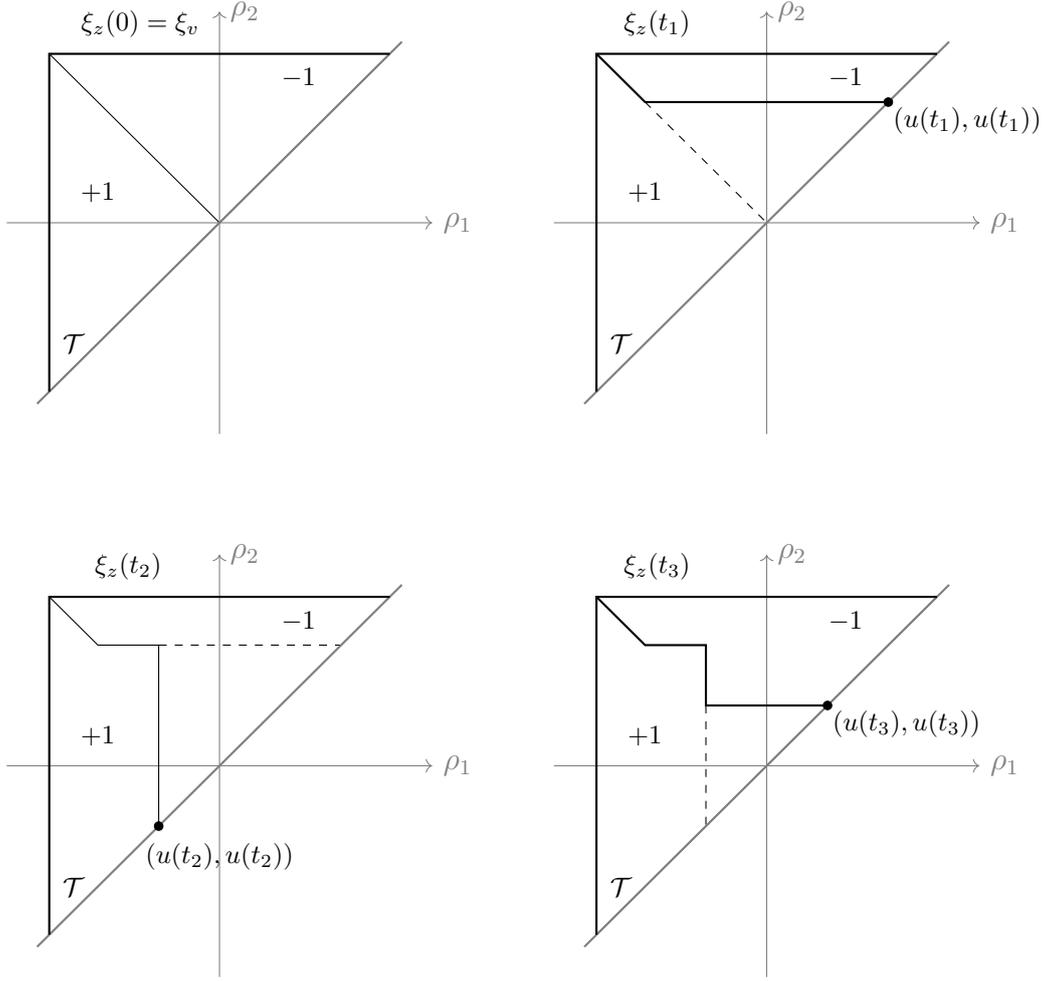
\begin{figure}[H]
    \centering
    \begin{tikzpicture}[scale=0.8]
    \draw[->, gray] (-3.5, 0) -- (3.5, 0) node[right] {$\p_1$}; 
    \draw[->, gray] (0, -3.5) -- (0, 3.5) node[right] {$\p_2$}; 

    \draw[black] (-2.8,-2.8)--(2.8,2.8);
    \draw[black] (-2.8,2.8)--(0,0);
    \draw[thick, gray] (-3,-3)--(3,3);
    \draw[thick ] (-2.8,-2.8)--(-2.8,2.8)--(2.8,2.8);
    
    \node at (-2,0.5) {\footnotesize{$+1$}};
    \node at (-1.3,3.3) {\footnotesize{$\xi_z(0)=\xi_v$}};
    \node at (1.3,2.4) {\footnotesize{$-1$}};
    \node at (-2.4,-2) {\footnotesize{$\T$}};

    \draw[->, gray, xshift =9 cm] (-3.5, 0) -- (3.5, 0) node[right] {$\p_1$}; 
    \draw[->, gray, xshift = 9 cm] (0, -3.5) -- (0, 3.5) node[right] {$\p_2$}; 

    \draw[black, xshift = 9 cm] (-2.8,-2.8)--(2.8,2.8);
    \draw[black, dashed, xshift = 9 cm] (-2,2)--(0,0);
    \draw[thick, gray, xshift = 9 cm] (-3,-3)--(3,3);
    \draw[thick, xshift = 9 cm] (-2.8,2.8)--(-2,2)--(2,2);
    \filldraw[xshift = 9cm] (2,2) circle (2pt);
    
    \draw[thick, xshift = 9 cm] (-2.8,-2.8)--(-2.8,2.8)--(2.8,2.8);
    
    \node at (7,0.5) {\footnotesize{$+1$}};
    \node at (7.2,3.3) {\footnotesize{$\xi_z(t_1)$}};
    \node at (12.3,1.7) {\footnotesize{$(u(t_1),u(t_1))$}};
    \node at (10.3,2.4) {\footnotesize{$-1$}};
    \node at (6.6,-2) {\footnotesize{$\T$}};

     \draw[->, gray, yshift = -9cm] (-3.5, 0) -- (3.5, 0) node[right] {$\p_1$}; 
    \draw[->, gray, yshift = -9cm] (0, -3.5) -- (0, 3.5) node[right] {$\p_2$}; 

    \draw[black, yshift =- 9cm] (-2.8,-2.8)--(2.8,2.8);
    \draw[black, yshift = -9cm] (-2.8,2.8)--(-2,2)--(-1,2)--(-1,-1);
    \draw[black, dashed, yshift = -9cm] (-1,2)--(2,2);
    \draw[thick, gray, yshift = -9cm] (-3,-3)--(3,3);
    \draw[thick, yshift = -9cm ] (-2.8,-2.8)--(-2.8,2.8)--(2.8,2.8);

    \filldraw[ yshift = -9cm] (-1,-1) circle (2pt);
    
    \node at (-2,-8.5) {\footnotesize{$+1$}};
    \node at (-1.5,-5.7) {\footnotesize{$\xi_z(t_2)$}};
    \node at (1.3,-6.6) {\footnotesize{$-1$}};
    \node at (-2.4,-11) {\footnotesize{$\T$}};
    \node at (0,-10.5) {\footnotesize{$(u(t_2),u(t_2))$}};

    \draw[->, gray, xshift =9 cm, yshift = -9cm] (-3.5, 0) -- (3.5, 0) node[right] {$\p_1$}; 
    \draw[->, gray, xshift = 9 cm, yshift = -9cm] (0, -3.5) -- (0, 3.5) node[right] {$\p_2$}; 

    \draw[black, xshift = 9 cm, yshift = -9cm] (-2.8,-2.8)--(2.8,2.8);
    \draw[black, dashed, xshift = 9 cm, yshift = -9cm] (-1,1)--(-1,-1);
    \draw[thick, gray, xshift = 9 cm, yshift = -9cm] (-3,-3)--(3,3);
    \draw[thick, xshift = 9 cm, yshift = -9cm] (-2.8,2.8)--(-2,2)--(-1,2)--(-1,1)--(1,1);
    \filldraw[xshift = 9cm, yshift = -9cm] (1,1) circle (2pt);
    
    \draw[thick, xshift = 9 cm, yshift = -9cm] (-2.8,-2.8)--(-2.8,2.8)--(2.8,2.8);
    
    \node at (7,-8.5) {\footnotesize{$+1$}};
    \node at (7.2,-5.7) {\footnotesize{$\xi_z(t_3)$}};
    \node at (11.3,-8.3) {\footnotesize{$(u(t_3),u(t_3))$}};
    \node at (10.3,-6.6) {\footnotesize{$-1$}};
    \node at (6.6,-11) {\footnotesize{$\T$}};
    \end{tikzpicture}
    \caption{A possible $\xi_z(t)$, where the input is piecewise monotone as described in the paragraph. Here, we represent $\xi_z(t_{i-1})$ by the dashed line together with $\xi_z(t_{i})$, represented by the solid line, to highlight the evolution in time of $\xi_z(\cdot)$;}
    \label{fig: prei2}
\end{figure}

Inspired by Remark \ref{rmk:scelta}, we state the following.

\begin{defi}
    \label{def:S} We denote by $\tilde{\mathcal S}$ the set of all $z\in \M$ such that the function $z:{\mathcal P}\to\{-1,1\}$ is generated starting from $z_v$ via an input $u\in C^0([0,T])$, that is, for $\rho\in{\mathcal P}$,

\[
z(\rho)=[h^\rho(u,z_v(\p))](T).
\] 
Any $z\in \M$ will be refereed as a configuration of the Preisach plane. Similarly, we denote by $\tilde{\mathcal B}$ the set of the corresponding maximal anti-monotone graphs $\{ \xi_z \,|\, z\in \tilde{\mathcal S}\}$ associated to a specific configuration. We also consider the metric space $(\tilde{\mathcal S}, d)$ setting \[d(z_1,z_2):= \int_\T |z_1-z_2| \, d\p,\] which is isometric to the metric space $(\tilde{\mathcal B}_z, d')$ with $d'(\xi_{z_1}, \xi_{z_2}):=d(z_1,z_2).$ We finally denote by $\tilde{\mathcal S}_L$ the set of the configurations $z\in\tilde{\mathcal S}$ generated by $L-$lipshitz inputs $u$, and consequently by $\tilde{\mathcal B}_L:=\{\xi_z ~|~ z \in \tilde{\mathcal S}_L\}.$
\end{defi}

\begin{rmk}
    $\tilde{\mathcal{S}}$ and $\tilde{\mathcal B}$ are not compact metric spaces, this is why we introduced $\tilde{\mathcal S}_L$ and $\tilde{\mathcal B}_L$ which can be proven to be compact subset of $\tilde{\mathcal S}$ and $\tilde{\mathcal B}$ respectively (see the continuity properties of the Preisach operator in \cite[Chapter~IV]{AVH}).
\end{rmk}

Similarly to Proposition \ref{prop: Relayweak}, the following characterization via measures for the Preisach operator holds.

\begin{prop} \label{prop: preisweak}
    Let us fix $z_0 \in \tilde{\mathcal S}$ and a couple of function $u,w: [0,T] \to \R$ piecewise constant with a finite number of discontinuities with $w=\int_\T z\,d\p$ for some $z:[0,T] \mapsto \tilde{\mathcal S}$. Moreover, denote by $\tilde{u}$ and $\tilde{z}$ the right-continuous representatives of $u$ and $z$ respectively and suppose $u(0)=\tilde{u}(0)$, $z(0)=\tilde{z}(0)=z_0$ and that $(u(0),z_0(\rho))\in{\cal \LL_\p}$ for a.e. $\p$. Then $(u,w)$ satisfies the Preisach relationship $w(t)=[\HM(u,z_0)](t)$ in $[0,T]$ if and only if \begin{enumerate}[1)]
        \item for almost every $t\in [0,T]$ and for any measurable subset $\T'\subset\T$, it holds 
        \begin{equation}\label{eq: preisweak1}\begin{split}
            \int_{\T'} \int_{0}^{t} \tilde{u} ~d(Dz(\cdot\,;\p))\, d\p&\geq \int_{\T'}\left(\int_{0}^{t} \p_2 ~d(D^+z(\cdot\,;\p))-\int_{0}^{t} \p_1 ~d(D^-z(\cdot\,;\p))\right) d\p\\&= \int_{\T'} \Psi_\p(z;(0,t))\,d\p
            \end{split}
        \end{equation}
        \item for almost every $t\in [0,T]$ and almost every $\p\in\T$\begin{equation}\label{eq: preisweak2}
         z(t;\p)=[h^\p(u,z_0(\p))](t).\end{equation}
    \end{enumerate}
\end{prop}
\begin{proof}    
``$(\implies)$" This implication follows directly from Proposition \ref{prop: Relayweak}, indeed the measure $\LM^2$ is positive.\par
    ``$(\impliedby )$" Now assume instead $1)$ and $2)$. If $u$ is constant in time, let's say $u\equiv u_0$, then from \eqref{eq: preisweak2} it follow that, on the Preisach plane, on $\{\p_2 \leq u_0\},$ $z(\cdot;\p)\equiv +1,$ and on $\{\p_1 \geq u_0\},$ $z(\cdot;\p)\equiv -1.$ Consider now the set $\T':=\{\rho=(\rho_1,\rho_2)|\p_1 < u_0,\, \p_2 >u_0\}$ we claim that $z\equiv z_0$ in $\T'$. 
    Suppose by contradiction that this is not true, that is, at a time $t^*$, $z(t^*+,\cdot)\not = z(t^*-,\cdot)$ on $\T'$, and consider $\T'_+:=\{\p\in \T' \,|\, z(t^*-,\rho)=-1,\, z(t^*+,\rho)=+1\} $ and $\T'_-:=\{\p \in \T'\,|\, z(t^*-,\rho)=+1,\, z(t^*+,\rho)=-1\}.$ By absurd hypothesis at least one among $\T'_-$ and $\T'_+$ must have positive Lebesgue measure. From \eqref{eq: preisweak1} applied for some $t>t^*$ and from the definition of $\T'$ we then get the contradicting inequalities \[u_0 (\LM^2(\T'_+)-\LM^2(\T'_-)) \geq \int_{\T'_+} \p_2 \, d\p - \int_{\T'_-} \p_1 \, d\p> u_0 (\LM^2(\T'_+)-\LM^2(\T'_-)).  \] Hence it must be $\LM^2(\T'_+)=\LM^2(\T'_-)=0$, which means that $z \equiv z_0$ and consequently $w(t)=w(0)=[\HM(u,z_0)](0)=[\HM(u,z_0)](t)$ in $[0,T]$. Now assume $u$ not constant with one jump discontinuity at $t^*$. Again on the set $\{\rho\in{\mathcal P}|\p_2 \leq u(t^*+)\}$ it is $z(t^*,\cdot)=1$, on $\{\p\in{\mathcal P}|\p_1 \geq  u(t^*+)\}$ it is $z(t^*,\cdot)=-1$, and it can be similarly proved as above that on $\T':=\{\p\in{\mathcal P}|\p_1 \leq  u(t^*+),\, \p_2 \geq  u(t^*+)\}$ $z$ does not switch: $z(t^*+,\cdot)=z(t^*-,\cdot)=z_0(t',\cdot)$ for any $t'<t^*$ such that $u$ has no jumps in the time interval $[t',t^*]$. This implies that $\xi_z(t^*+)$ defined as in \eqref{eq:csi} is exactly  the maximal anti-monotone graph associated to the process $\HM(u,z_0)$ and hence $w=\HM(u.z_0)$ in $[0,T]$. 
\end{proof}

We now prove that for an input $u$ piecewise constant the state $z$ generated is in $\tilde{\mathcal S}_L$ for some suitable $L$ depending on the total variation of $u$. To do so, following \cite[Section~III.6]{AVH}, we introduce the so called reduced memory sequence of $u$. Define \[\begin{split}   
M:= \max_{[0,T]} u, &\quad m := \min_{[0,T]} u, \\ \quad t_M:= \sup \{ t\in [0,T]\,|\, u(t)=M\} \quad & \text{and} \quad t_m:= \sup \{ t\in [0,T]\,|\, u(t)=m\}.\end{split}\] Excluding the trivial case when $t_m=t_M$, suppose for example that $t_M < t_m\leq T$, the other case is quite similar. Then we set  \[ t_1=t_M,\quad \eta^1=M, \quad \eta_2:=\min_{[t_1,T]} u  \quad \text{and} \quad t_2:= \sup\{ t\in ]t_1,T]\,|\, u(t) = \eta_2\},  \] and if $t_2= T$ we stop. If not we set \[\eta^3:=\max_{[t_2,T]} u  \quad \text{and} \quad t_3:= \sup\{ t\in ]t_2,T]\,|\, u(t) = \eta^3\}\] and we go on alternating local maxima and minima until $t_k=T$. Notice that the procedure always ends as we suppose the function to have a finite number of constant pieces. Then the resulting sequence $\{\eta^1,\eta_2,\eta^3,\eta_4,\dots\}$ (or $\{\eta_1,\eta^2,\eta_3,\eta^4,\dots\}$ if we start with a minimum) is called reduced memory sequence, denoted by $RMS(u; [0,T])$ and it gives information about the sequence of maximal peaks (resp. minimal peaks) not deleted by subsequent greater peaks (resp. lower peaks). Such a sequence is important for the Preisach operator, indeed the output of the operator with input $u$ depends only on the $RMS(u; [0,T])$, that is given two inputs $u_1,u_2$ either continuous or piecewise constant, if $RMS(u_1; [0,T])=RMS(u_2; [0,T])$, then $w_1=[\HM(u_1,z_0)]=[\HM(u_2,z_0)]=w_2$.\par Exploiting this property we can give the criterion for a piecewise constant input to generate configuration $z \in \tilde{\mathcal S}_L.$ We first need the following lemma.

\begin{lem}\label{lemma: propPreis1}
    Consider a piecewise constant function $u:[0,T]\to\R$, consisting of a finite number of pieces, with its total variation $Var(u)\leq C$. Then there exists a function $u_{RMS}:[0,T]\to \R$ that is $L$-Lipschitz with $L$ depending on $C$ and $T$ such that $u_{RMS} (0)=u(0)$, $RMS(u;[0,T])=RMS(u_{RMS};[0,T])$ and $Var(u_{RMS})\leq C$.
\end{lem}
\begin{proof}
    Consider the reduced memory sequence of $u$, which will be of the form $\{\eta^1,\eta_2,\eta^3,\dots\}$ or $\{\eta_1,\eta^2,\eta_3,\dots\}$, where $\eta^i$ are maxima and $\eta_i$ are minima. We know that, denoting the elements of RMS by $v_i$ with $v_0=u(0)$, \[\sum_{i=1} |v_{i}-v_{i-1}|\leq Var(u)\leq C.\]
    Now define the sequence \[z_j:= \sum_{i=1}^{j}|v_i-v_{i-1}|\] so that $\{z_j\}_j$ is a monotonically increasing sequence bounded by $\tilde{C}\leq C$, with $\tilde{C} = \max z_j$ and set $z_0=u(0).$ By rescaling the points to the interval $[0,T]$, that is by setting
    \[
    x_i:= \frac{T}{\Tilde{C}} z_i,
    \] 
    we can define the function $u_{RMS}$ as the piecewise linear function that interpolates the points $(x_i,v_i)$. \par 
    It is clear that $RMS(u;[0,T])=RMS(u_{RMS};[0,T])$ and $Var(u_{RMS})\leq C$. \par
    Moreover, considering the modulus $m_i$ of the slopes of the segments corresponding to the linear parts we have that
    \[m_i = \frac{|v_i-v_{i+1}|}{|x_{i+1}-x_i|} = \frac{\tilde{C}}{T} \frac{|v_i-v_{i+1}|}{|z_{i+1}-z_i|} = \frac{\tilde{C}}{T} \leq \frac{C}{T},\]
    hence $u_{RMS}$ is $L-$lipshitz with $L\leq C/T$.
\end{proof}

\begin{prop}\label{prop: Ltilda}
Let $L>0$ and $z_0\in\tilde{\mathcal S}_L$ and $u:[0,T]\to\mathbb{R}$ piecewise constant, with $Var(u)\le C$ and $u$ compatible with $z_0$. Then, for every $\tau\in[0,T]$, the function $z(\tau):\rho\mapsto [h^\rho(u,z_0(\rho))](\tau)$ belongs to $\tilde{\mathcal S}_{2\max(L,C/T)}$.
\end{prop}

\begin{proof}
By Definition \ref{def:S}, there exists a $L$-Lipschitz function $u^0:[0,T]\to\mathbb{R}$ such that $z_0(\rho)=[h^\rho(u^0,z_v(\rho))](T)$. We define $\tilde{u}:[0,T] \to \R$ as \[\tilde{u}(t):=\begin{cases}
    u^0(t) \quad & t\in \left[0,\frac{T}{2}\right],\\
    (u^{\tau})_{RMS}(2t-T) \quad & t\in \left[\frac{T}{2}, T\right],
\end{cases}\] where $u^{\tau}:[0,T]\to \R$ is \[u^\tau(t):=\begin{cases}
    u(t) \quad& t\leq \tau,\\
    u(\tau) \quad& t>\tau.
\end{cases}\] The conclusion follows from Lemma \ref{lemma: propPreis1} as $[h^\rho(u,z_0(\rho))](\tau) = [h^\rho(\tilde{u},z_v(\rho))](T)$.

\end{proof}

We define now $\tilde{\mathcal B}_L^{(n)}$ the set of graphs of $\tilde{\mathcal B}_L$ consisting of a finite number of segments, with vertices belonging to the set $Q_n:=(A_n\times A_n) \cap \T$ with $A_n=\{k 2^{-n} ~|~k\in \Z\}$ and $\tilde{\mathcal S}_L^{(n)}$ is the set of corresponding configurations. This set will be used when constructing the wave-front tracking approximation in Section \ref{S5}. 

\begin{lem}\label{lemma: propPreis4}
    Let $\xi_z\subset \tilde{\mathcal B}_L$ be an antimonotone graph composed by a finite number of horizontal and vertical segments. Then for every $n\in\N$ there exists an antimonotone graph $\xi_{z_n} \in \tilde{\mathcal B}_L^{(n)}$ such that \begin{equation}
        d(\xi_z,\xi_{z_n})\leq c_1 2^{-n},
    \end{equation} with $c_1$ depending only on the triangle $\T$. 
\end{lem}
\begin{proof}
    Consider the reduced memory sequence $\{\eta^1,\eta_2,\dots\}$ (or $\{\eta_1,\eta^2,\dots\}$) of the input $u$ that generates $\xi_z$. Now redefine $\tilde{\eta}^i:=\max\{ t \in A_n\,|\, t\leq \eta^i\}$ and $\tilde{\eta}_i:=\min\{ t \in A_n\,|\, t\geq \eta_i\}$ and consider the new sequence $\{\tilde{\eta}^1,\tilde{\eta}_2,\dots\}$ (or $\{\tilde{\eta}_1,\tilde{\eta}^2,\dots\}$). Observe that, at least for $n$ small, $\{\tilde{\eta}^1,\tilde{\eta}_2,\dots\}$ might not be of the form of a reduced memory sequence as it could happen that $\tilde{\eta} ^{i} \leq \tilde{\eta} _{i+1}$ or $\tilde{\eta} _{i} \geq \tilde{\eta} ^{i+1}$ for some $i$, but in that case $|\eta ^j - \eta _j| \leq 2^{-n}$ for $j\geq i$ so we could simply truncate the new sequence at $\tilde{\eta}^i$ or $\tilde{\eta}_i$. Modifying the initial input $u$, it is easy to construct the input $u_n$ that is still $L-$Lipschitz with reduced memory sequence $\{\tilde{\eta}^1,\tilde{\eta}_2,\dots\}$ (or $\{\tilde{\eta}_1,\tilde{\eta}^2,\dots\}$), that hence generates a configuration $z_n \in \tilde{\mathcal S}_L^{(n)}$ and graph $\xi_{z_n}\in\tilde{\mathcal B}_L^{(n)}$. By denoting the vertices $v_1,\dots, v_k$ of $\xi_z$ it is easy to see that the vertices of $\xi_{z_n}$ belong to the strip \[ S_n := \bigcup_{(x',y')\in \xi_z} \{ (x,y) ~|~ |x-x'|\leq 2^{-n}, |y-y'|\leq 2^{-n}\}\cap \T\] hence \begin{equation*} d(\xi_z,\xi_{z_n})\leq 2~|S_n|\leq 4 ~ 2^{-n}(a+2^{-n}) \leq 4  (a+1) 2^{-n}=c_1 2^{-n},
    \end{equation*} see also Figure \ref{fig: regioneSn}.
\end{proof}

    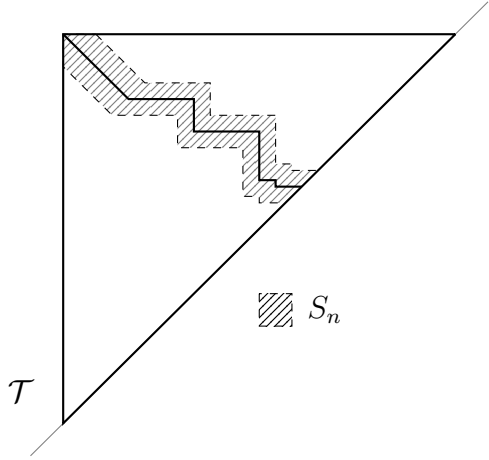
\begin{figure}
        \centering
        \begin{tikzpicture}[scale =0.43]
        
         \filldraw[ dashed, thin,pattern=north east lines, pattern color=black!50] (-5,6)--(-3.5,4.5)--(-1.5,4.5)--(-1.5,3.5)--(0.5,3.5)--(0.5,2)--(1,2)--(1,1.8)--(1.8,1.8)--(0.8,0.8)--(0,0.8)--(0,1)--(-0.5,1)--(-0.5,2.5)--(-2.5,2.5)--(-2.5,3.5)--(-4.5,3.5)--(-6,5)--(-6,6)--cycle;
    
        \draw[-,gray] (7,7)--(-7,-7);
        \draw[-,thick] (6,6)--(-6,-6)--(-6,6)--(6,6);
        \draw[-, thick] (-6,6)--(-4,4)--(-2,4)--(-2,3)--(0,3)--(0,1.5)--(0.5,1.5)--(0.5,1.3)--(1.3,1.3);
        \filldraw[dashed,thin ,pattern=north east lines, pattern color=black] (0,-2)--(0,-3)--(1,-3)--(1,-2)--cycle;
        \node at (2,-2.5) {$S_n$};
        \node at (-7.3,-5) {$\T$};
         
        \end{tikzpicture}
       
        \caption{Referring to the proof of Lemma \ref{lemma: propPreis4}, we highlight, inside the triangle $\T$, the region $S_n$ where $\xi_{z_n}$ belongs to. The filled graph is the original  graph $\xi_z$.}
         \label{fig: regioneSn}
    \end{figure}

\begin{rmk}\label{rmk: state}
When we introduce the operator $\HM$ in a PDE, we have to define it on functions $u$ both depending on a space variable $x\in\R^n$, for some $n$, and on time $t$. In this case, for every fixed $x\in\R^n$, we see $u(x,\cdot)$ as a function of time only, and then we define the output as \begin{equation}
     w(x,t):= [\HM(u(x,\cdot),z_0(x))](t), \quad \text{a.e}\ x, \forall\ t,
\end{equation} 

\noindent
where the initial configuration of the output is now a given function depending on $x$.
\end{rmk}
    

\section{Weak formulation for the Cauchy problem}\label{S3}

We present and explain the weak formulation for our problem

\begin{equation}\label{eq: hpde}
    \begin{cases}
        u_t+w_t+u_x=0\quad &\text{in } \R \times [0,T),\\
        w=[\HM(u,z_0)] \quad &\text{in } \R \times [0,T),\\
        u(x,0)=u_0(x) \quad &\text{in } \R ,\\
        z_0(x) \in \tilde{\mathcal S}_L\quad &\text{in } \R,
    \end{cases}
\end{equation} where the state-dependent hysteresis operator is defined as in Remark \ref{rmk: state}, $u_0$ and $z_0$ are compatible for each $x\in \R$ and $L>0$.

\begin{defi}\label{def: weak}
    A couple of functions $(u,w)$ with $u,w \in C^0([0,T];L^1_{loc}(\R))$ is a weak solution to \eqref{eq: hpde} if: \begin{enumerate}[i)]
    \item it satisfies the following weak formulation of the PDE \begin{equation}\label{eq: hweaksol}
        \int\limits_0^{+\infty} \int\limits_{-\infty }^{+\infty}[(u+w) \phi_t+u \phi_x] ~dx~ dt +\int\limits_{-\infty}^{+\infty} (u_0(x) +w_0(x))\phi(x,0) ~dx=0,
    \end{equation} for every $C^1$ function $\phi$ with compact support in $\R \times [0,T)$, with $w_0(x)= \int_\T z_0(x)\,d\p;$
    \item there exists a function $z: \R \times [0,T) \to \tilde{\mathcal S}$ such that for almost every $(x,t)$ and almost every $\p=(\p_1,\p_2)\in{\cal P}$ we have \begin{equation}\label{eq: def11}
        w(x,t) = \int_{\T} z (x,t;\p) ~d\p,
    \end{equation}$|z(x,t;\p)|=1$ and \begin{equation}\label{eq: hdis}
            \begin{cases}
                (z(x,t; \p)-1)(u-\p_2) \geq 0 \\
                (z(x,t;\p)+1)(u-\p_1) \geq 0.
            \end{cases}
        \end{equation}  
    \item for almost every $x$ and $\p$, the distributional derivative $\frac{\partial z }{\partial t} (x, \cdot; \p)$ is a measure on $[0,T)$ (denoted in the same way) that satisfies \begin{equation}
        \label{eq: genweakhis}
        \frac{1}{2}\int_\R (u(x,t)^2-u_0(x)^2)~dx+\int_\R \int_\T \Psi_\p(z; (0,t)) ~d\p dx \leq 0
        \end{equation}  for almost every $t\in (0,T)$, where $\Psi_\p(z;(0,t))$ is as defined in \eqref{eq: weakhis}.
\end{enumerate}
\end{defi}
\begin{rmk} Condition i) is the standard weak integral formulation of the PDE. Condition ii) is the requirement that $w$ is the output of the Preisach operator with input $u$, being $z(\cdot,\cdot;\rho)$ the corresponding output of the Relay with threshold $\rho$. In condition iii), similarly to \cite{AVH1} and as in \cite{BFMS}, equation \eqref{eq: genweakhis} can be interpreted as an equivalent formulation of \eqref{eq: preisweak1}, in the case of $H^1$ functions which also depend on the state $x$. Indeed suppose that both $u, w$ are in $H^1(\R \times (0,T))$  solution of the PDE \eqref{eq: hpde} then \eqref{eq: genweakhis} reads as follows: \begin{equation}
        \int_\R \int_0^t u u_t~dtdx+\int_\R \int_\T \Psi_\p(z; (0,t)) ~d\p dx \leq 0
    \end{equation} which, by the PDE, \begin{equation}
        \int_\R \int_0^t u (-w_t-u_x)~dtdx+\int_\R \int_\T \Psi_\p(z; (0,t)) ~d\p dx \leq 0.
    \end{equation} Now because of \eqref{eq: def11} and since $u$ has compact support then we can conclude that \begin{equation}
        -\int_\R  \int_\T \left[ \int_0^t u\, z_t - \Psi_\p(z; (0,t))\right] ~d\p dx \leq 0,
    \end{equation} which is \eqref{eq: preisweak1} integrated on the whole domain $\T$ and generalized to the space dependence.
\end{rmk}

If the weak solution has a jump discontinuity on a curve $(s(t),t)$, e.g. say $u_-(t)\not= u_+(t)$ or $w_-(t)\not=w_+(t)$ then we get, in a standard way, the following (extended) \RH condition \begin{equation}\label{eq: hrhcond}
    \frac{u_-(t)-u_+(t)}{u_-(t)-u_+(t)+w_-(t)-w_+(t)} = s'(t).
\end{equation}

\noindent
In Section \ref{S6} we will also introduce an entropy condition which will ensure uniqueness.


\section{The Riemann problem}\label{S4}

We deal now with the Riemann problem for \eqref{eq: hpde}, thus we set initial conditions 
\begin{equation}\label{eq: datiR}
u_0= \begin{cases}
    u_l \quad x<0,\\
    u_r \quad x>0,
\end{cases} \quad z_0= \begin{cases}
    z_l \quad x<0,\\
    z_r \quad x>0.
\end{cases}
\end{equation}
We also suppose $z_l,z_r\in\tilde{\cal S}_L$ for some $L>0$, with $(u_0(x),z_0(x;\p))\in \LL_\p$ for a.e. $\p\in\T,$ and that $\xi_l,\xi_r\in \tilde{\mathcal B}_L$, the graphs associated to $z_l$ and $z_r$ respectively, consist of a finite collection of vertical and horizontal segments.
\par \textit{\textbf{Case 1: $\mathbf{u_l>u_r}.$}}
For fixed $x>0$, we are expecting the solution $u(x,\cdot)$ to increase from $u_r$ to $u_l$ as the discontinuity of the initial data propagates to the right. 
Let us denote by $M_i$ the ordinate of the horizontal segments of $\xi_r$ strictly above the ordinate $u_r$ and $m_i$ the abscissa of the vertical ones (see Figure \ref{fig: segmenti}). We order them in the following way: $M_1< \dots < M_N$ and $m_1> \dots > m_N$. For compatibility reasons we either have $u_r = m_1$ if the last segment of $\xi_r$ is vertical or $u_r=M_0$ for some $M_0< M_1$ if the last segment of $\xi_r$ is horizontal. 

\begin{figure}
        \centering
        \begin{tikzpicture}[scale =0.43]
         \begin{scope}
            
            \draw[dashed,gray] (-2,4)--(5,4) node[right]{$M_3$};
            \draw[dashed,gray] (0,3)--(5,3) node[right]{$M_2$};
            \draw[dashed,gray] (0.5,1.5)--(5,1.5) node[right]{$M_1$};
    
            \draw[dashed,gray] (-2,3)--(-2,-5) node[below]{$m_3$};
            \draw[dashed,gray] (0,1.5)--(0,-5) node[below]{$m_2$};
            \draw[dashed,gray] (0.5,1.3)--(0.5,-5) node[below right]{$m_1$};
        
            \draw[-,gray] (7,7)--(-7,-7);
            \draw[-,thick] (6,6)--(-6,-6)--(-6,6)--(6,6);
            \draw[-] (-6,6)--(-4,4)--(-2,4)--(-2,3)--(0,3)--(0,1.5)--(0.5,1.5)--(0.5,1)--(1,1) node[below right]{$(u_r,u_r)$};

            \filldraw[black] (1,1) circle (2.5pt);
            
        \end{scope}
            \begin{scope}[xshift = 18cm]
        
            \draw[dashed,gray] (-2,4)--(5,4) node[right]{$M_3$};
            \draw[dashed,gray] (0,3)--(5,3) node[right]{$M_2$};
            \draw[dashed,gray] (0.5,1.5)--(5,1.5) node[right]{$M_1$};
    
            \draw[dashed,gray] (-2,3)--(-2,-5) node[below]{$m_3$};
            \draw[dashed,gray] (0,1.5)--(0,-5) node[below]{$m_2$};
            \draw[dashed,gray] (0.5,0.5)--(0.5,-5) node[below right]{$m_1$};
        
            \draw[-,gray] (7,7)--(-7,-7);
            \draw[-,thick] (6,6)--(-6,-6)--(-6,6)--(6,6);
            \draw[-] (-6,6)--(-4,4)--(-2,4)--(-2,3)--(0,3)--(0,1.5)--(0.5,1.5)--(0.5,0.5) node[below right]{$(u_r,u_r)$};

            \filldraw[black] (0.5,0.5) circle (2.5pt);
            
        \end{scope}
        
        \end{tikzpicture}
        \caption{Here we highlight the notation used when dealing with the Riemann problem; on the left the case when the last segment of $\xi_r$ is horizontal, on the right when it is vertical.}
        \label{fig: segmenti}
\end{figure}
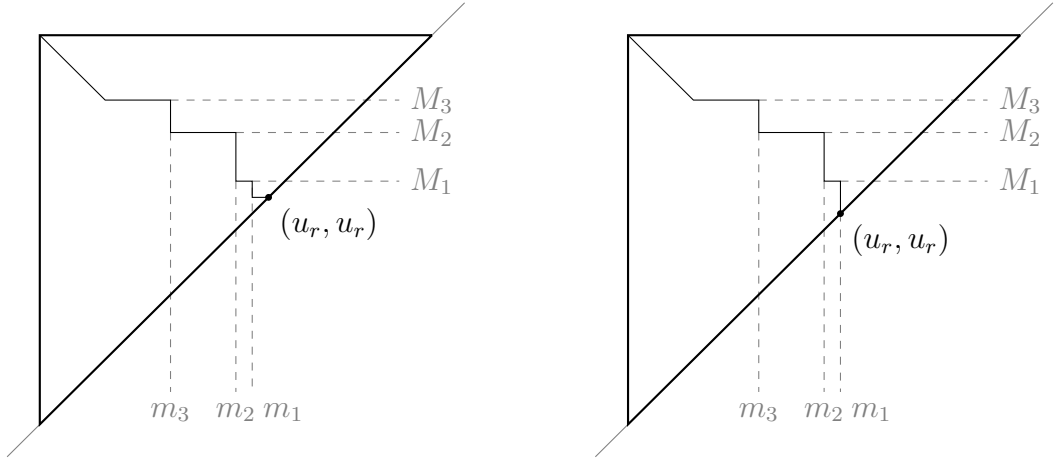

\par While $u(x,\cdot)$ increases from $u_r$ to $u_l$, the graph $\xi(x,\cdot)$ will move up as explained in Section \ref{S2}. Assume, for the moment, that $t\mapsto u(x,t)$ is absolutely continuous. For $u_r \leq u(x,\cdot) \leq M_1$ we have that $w(x,\cdot)$ increases from $w_r$ adding a quantity equal to twice the area of a trapezoid if $u_r=M_0$ (or of a triangle if $u_r=m_1$). Hence the time derivative $w_t(x,\cdot)$, while $u\leq M_1$, is equal to twice the length of an horizontal segment, which is $u(x,\cdot)-m_1$, times $u_t(x,\cdot)$. 
Using the same reasoning for every $M_i$ such that $M_i\leq u_l$, if $u_l\leq M_k$ we get that \begin{equation}
    w_t = \begin{cases} 
        2(u-m_1) u_t &\quad u_r < u < M_{1}\\
        2(u-m_2) u_t &\quad M_1 < u < M_2 \\
        \dots \\
        2(u-m_i)u_t &\quad M_{i-1} < u< M_i \\
        \dots \\
        2(u-m_k)u_t &\quad M_{k-1} < u < u_l \leq M_k.
    \end{cases}
\end{equation}
So we can rewrite the equation as follows \begin{equation}\label{eq: pdeg}
    u_t g(u) + u_x = 0,
\end{equation} with 
\begin{equation}
    g(u) = \begin{cases}
        1+2(u-m_1)&\quad u_r < u < M_{1}\\
        1+2(u-m_2)  &\quad M_1 < u < M_2 \\
        \dots \\
        1+2(u-m_i) &\quad M_{i-1} < u< M_i \\
        \dots \\
        1+2(u-m_k) &\quad M_{k-1} < u < u_l \leq M_k.
    \end{cases}
\end{equation}
By classic theory on conservation laws the solution to such PDE consists of a finite number of rarefaction waves. Indeed, \eqref{eq: pdeg} can be rewritten as follows (see \cite[Section~3.1]{DAFLIBRO} for equivalence of the two formulations in the case of non-classical solutions),
\begin{equation}\label{eq: flux}
    u_t+f(u)u_x=0,
\end{equation} where $f=1/g$ is piecewise decreasing, hence the flux is piecewise concave. Since $u_l>u_r$, by computing $f^{-1}(x/t)$ we get, \begin{equation}\label{eq: uRP}
    u(x,t) = \begin{cases}
        u_l & \quad \frac{x}{t}< \frac{1}{1+2(u_l-m_k)} \\ 
        \frac{1}{2}\left(\frac{t}{x}-1\right)+m_k  &\quad \frac{1}{1+2(u_l-m_k)} < \frac{x}{t} < \frac{1}{1+2(M_{k-1}-m_k)}\\
        M_{k-1} &\quad \frac{1}{1+2(M_{k-1}-m_k)} < \frac{x}{t} < \frac{1}{1+2(M_{k-1}-m_{k-1})}\\
        \frac{1}{2}\left(\frac{t}{x}-1\right)+m_{k-1}  &\quad \frac{1}{1+2(M_{k-1}-m_{k-1})} < \frac{x}{t} < \frac{1}{1+2(M_{k-2}-
m_{k-1})}\\
        \dots\\
        \dots\\
        M_1 &\quad \frac{1}{1+2(M_1-m_2)} < \frac{x}{t} < \frac{1}{1+2(M_1-m_1)}\\
        \frac{1}{2}\left(\frac{t}{x}-1\right)+m_1  &\quad \frac{1}{1+2(M_1-m_1)} < \frac{x}{t} < \frac{1}{1+2(u_r-m_1)}\\
        u_r &\quad \frac{1}{1+2(u_r-m_1)}<\frac{x}{t}
    \end{cases}
\end{equation}

\begin{figure}
        \centering
        \begin{tikzpicture}[scale =0.43]
        
        \draw[dashed,gray] (0.5,1.5)--(5,1.5) node[right,  yshift = -3pt]{$M_1$};

        \draw[dashed,gray] (0.5,0.5)--(0.5,-5) node[below]{$m_1$};

        \node[anchor = west] at (-2,2) {$\xi_r$};
        \node[anchor = west] at (1,3.2) {$\xi_{z^*}$};
    
        \draw[-,gray] (7,7)--(-7,-7);
        \draw[-,thick] (6,6)--(-6,-6)--(-6,6)--(6,6);
        \draw[-] (-6,6)--(-4,4)--(-2,4)--(-2,3)--(0,3)--(0,1.5)--(0.5,1.5)--(0.5,0.5) node[below right]{$(u_r,u_r)$};
        \draw[dash dot] (-6,6)--(-4,4)--(-2,4)--(-2,3)--(0,3)--(0,2.5)--(2.5,2.5);

        \filldraw[black] (2.5,2.5) circle (2.5pt) node[ yshift = -3pt, right]{$(u_l,u_l)$};

        \filldraw[black] (0.5,0.5) circle (2.5pt);

        \draw[gray, ->, xshift= 16cm, yshift = -6cm] (0,-1)--(0,13);
        \draw[gray, ->, xshift= 16cm, yshift = -6cm] (-2,0)--(6,0);
        \draw[-,xshift= 16cm, yshift = -6cm] (0,0)--(2,12);
        \draw[dashed, xshift= 16cm, yshift = -6cm] (0,0)--(0,12);
        \draw[-,xshift= 16cm, yshift = -6cm] (0,0)--(3,12);
        \draw[-,xshift= 16cm, yshift = -6cm] (0,0)--(4,12);
        \draw[-,xshift= 16cm, yshift = -6cm] (0,0)--(5,5);
        \fill[xshift= 16cm, yshift = -6cm,pattern=north east lines, pattern color=black, line width=0.5mm] (0,0)--(5,5)--(5,12)--(4,12);
        \fill[xshift= 16cm, yshift = -6cm, pattern=north east lines, pattern color=black, line width=0.5mm] (0,0)--(2,12)--(3,12);

        \fill[xshift= 18cm, yshift = -6cm,pattern=north east lines, pattern color=black, line width=0.5mm] (5,3)--(5,4)--(4,4)--(4,3)--(5,3);
        \draw[black,xshift= 18cm, yshift = -6cm] (5,3)--(5,4)--(4,4)--(4,3)--(5,3) node[above right]{rarefaction};
        \node[anchor = west] at (19,-4) {$(u_r,z_r)$};
        \node[anchor = west] at (12,0) {$(u_l,z_l)$}; 
        \node[anchor = west] at (15.65,6.2) {\tiny{$(u_l,z^*)$}};
        
        \end{tikzpicture}
        \caption{ Here we denote by $z^*$ the final configuration obtained starting from $z_r$ with the input $u$ increasing from $u_r$ to $u_l$; on the left the relationship between $u_l,u_r,\xi_{r}$ and $\xi_{z^*}$ (represented by the dot-dashed line); on the right the couple $(u,z)$ solution to the Riemann problem consisting of a union of rarefaction waves, separated by constant pieces and a possible vertical discontinuity for $z$.}
        \label{fig: rarefazioni}
\end{figure}
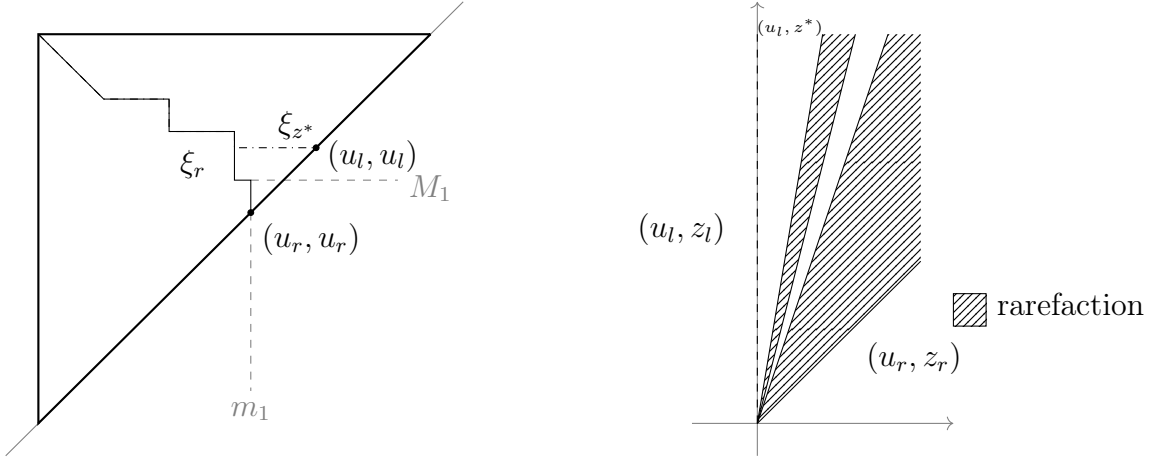
Once we get $u$, $z(x,t;\p)=[h^{\p}(u(x,\cdot),z_0(x;\p))](t)$ and  $w=[\HM(u,z_0)]$, can be easily computed and both will be given by a finite number of rarefaction waves, constant pieces and a possible discontinuity at $\{x=0\}$. See Figure \ref{fig: rarefazioni} for an example of a solution to the Riemann problem.

\begin{rmk}
    By setting $ w = \int_\T z\,d\p$, the couple $(u,w)$ constructed above satisfies the weak formulation \eqref{eq: hweaksol} associated to \eqref{eq: hpde}, \eqref{eq: datiR} for any $T>0$. Moreover, the relationship $w = [\HM(u,z_0)]$ holds strongly for $x\not=0,$ as $u(x,\cdot)$ is continuous.
\end{rmk}

If instead $u_l>M_N$ then length of the segment when $M_N\leq u\leq u_l$ is equal to $2u$ hence the change of area is $w_t=4u_t$. Hence, with a slight change in $w_t$, 
\begin{equation}
    w_t = \begin{cases}
        2(u-m_1) u_t &\quad u_r < u < M_{1}\\
        2(u-m_2) u_t &\quad M_1 < u < M_2 \\
        \dots \\
        2(u-m_i)u_t &\quad M_{i-1} < u< M_i \\
        \dots \\
        4uu_t &\quad M_{N} < u < u_l,
    \end{cases}
\end{equation} we can compute similarly as above $u,z$ and $w$. 

\textbf{\textit{Case $\mathbf{u_l<u_r}$}}. This case is symmetric to the previous one. Indeed, by following the previous reasoning, one can show that in this case $f$ appearing in \eqref{eq: flux} is instead piecewise increasing. However, since $u_l<u_r$, the resolution of the Riemann problem yields to a solution $u$ (consequently also $z$ and $w$) still composed by a finite number of rarefaction wave.

The following lemma is used to prove next Proposition \ref{prop: vartot}.
\begin{lem}\label{lemma: zvar1}
    Consider an initial configuration $z_0\in \tilde{\mathcal S}$ and a monotone evolution $u : [0,T]\to \R$. Given $z(\cdot):[0,T]\to \tilde S$ the configuration associated to the evolution $z(\cdot\,;\p)=[h^\p(u,z_0)](\cdot)$, then \begin{equation}
        {Var}_{[0,T]}(z) = d(z(T),z_0),
    \end{equation} where ${Var}_{[0,T]}(z)$, is as defined in Definition \ref{def: zL1BV}.
\end{lem}
\begin{proof}
    It is sufficient to prove that for every $t_1<t_2<t_3$ we have that \[d(z(t_3),z(t_1))=d(z(t_3),z(t_2))+d(z(t_2),z(t_1)).\] Suppose for example that $u$ is increasing. Denote by $u_i$ and $z_i$ respectively $u(t_i)$ and $z(t_i)$ for $i=1,2,3.$ We have that $z_1,z_2$ differ only in the strip $R_1=[-a,u_2]\times [u_1,u_2] \cap \T$, see Figure \ref{fig: lemmavar}. Hence \[d(z_1,z_2)=\int_{R_1} |z_1-z_2| ~d\p.\] We also have a region $R_2=[-a,u_3]\times[u_2,u_3] \cap \T,$ on which $z_2$ differs from $z_3$. Moreover $z_1$ and $z_3$ differ on $[-a,u_3]\times[u_1,u_3]\cap \T $ which is the disjoint union of $R_1$ and $R_2$ because of the monotonicity of $u$ and semigroup property \eqref{eq: semigroup}, see again \ref{fig: lemmavar}. Hence \[\begin{split} d(z_1,z_3)&=\int_{R_1} |z_1-z_3| ~d\p + \int_{R_2} |z_1-z_3| ~d\p \\&= \int_{R_1} |z_1-z_2| ~d\p + \int_{R_2} |z_2-z_3| ~d\p=d(z_1,z_2)+d(z_2,z_3).\end{split}\] 
\end{proof}

\begin{figure}
        \centering
        \begin{tikzpicture}[scale=0.23]

        \node[above] at (-6.2,6) {$z_1$};
        \node[above] at (9.8,6) {$z_2$};
        \node[above] at (25.8,6) {$z_3$};

        \draw[-, gray] (1.3,1.3)--(-6,1.3)--(-6,2)--(2,2)--cycle;
        \fill[pattern=north east lines, pattern color=gray, line width=0.5mm] (1.3,1.3)--(-6,1.3)--(-6,2)--(2,2)--cycle;
        \draw[-,gray] (3.5,3.5)--(-6,3.5)--(-6,2)--(2,2)--cycle;
        \fill[pattern=horizontal lines, pattern color=gray, line width=0.5mm] (3.4,3.4)--(-6,3.4)--(-6,2)--(2,2)--cycle;
    
        \draw[-,gray] (7,7)--(-7,-7);
        \draw[-,thick] (6,6)--(-6,-6)--(-6,6)--(6,6);
        \draw[-] (-6,6)--(-4,4)--(-2,4)--(-2,3)--(0,3)--(0,1.5)--(0.5,1.5)--(0.5,1.3)--(1.3,1.3) node[below right]{$(u_1,u_1)$};
        \filldraw[black] (1.3,1.3) circle (3.5pt);
        
        \draw[-,gray,xshift=16cm] (1.3,1.3)--(-6,1.3)--(-6,2)--(2,2)--cycle;
        \fill[pattern=north east lines, pattern color=gray, line width=0.5mm,xshift=16cm] (1.3,1.3)--(-6,1.3)--(-6,2)--(2,2)--cycle;
        \draw[-,gray,xshift=16cm] (3.5,3.5)--(-6,3.5)--(-6,2)--(2,2)--cycle;
        \fill[pattern=horizontal lines, pattern color=gray, line width=0.5mm,xshift=16cm] (3.4,3.4)--(-6,3.4)--(-6,2)--(2,2)--cycle;
        \draw[-,gray,xshift=16cm] (7,7)--(-7,-7);
        \draw[-,thick,xshift=16cm] (6,6)--(-6,-6)--(-6,6)--(6,6);
        \draw[-,xshift=16cm] (-6,6)--(-4,4)--(-2,4)--(-2,3)--(0,3)--(0,2)--(2,2) node[below right]{$(u_2,u_2)$};
        \filldraw[black,xshift=16cm] (2,2) circle (3.5pt);

        \draw[-,gray,xshift=32cm] (1.3,1.3)--(-6,1.3)--(-6,2)--(2,2)--cycle;
        \fill[pattern=north east lines, pattern color=gray, line width=0.5mm,xshift=32cm] (1.3,1.3)--(-6,1.3)--(-6,2)--(2,2)--cycle;
        \draw[-,gray,xshift=32cm] (3.5,3.5)--(-6,3.5)--(-6,2)--(2,2)--cycle;
        \fill[pattern=horizontal lines, pattern color=gray, line width=0.5mm,xshift=32cm] (3.4,3.4)--(-6,3.4)--(-6,2)--(2,2)--cycle;
        \draw[-,gray,xshift=32cm] (7,7)--(-7,-7);
        \draw[-,thick,xshift=32cm] (6,6)--(-6,-6)--(-6,6)--(6,6);
        
        \draw[-,xshift=32cm] (-6,6)--(-4,4)--(-2,4)--(-2,3.5)--(3.5,3.5) node[below right]{$(u_3,u_3)$};
        \filldraw[black,xshift=32cm] (3.5,3.5) circle (3.5pt);

        \draw[black,xshift=32cm] (3,-3.3) circle (25pt) node[right]{$\, R_1$};
        \fill[pattern=north east lines, pattern color=gray, line width=0.5mm,xshift=32cm] (3,-3.3) circle (25pt);

        \draw[black,xshift=32cm] (3,-6) circle (25pt) node[right]{$\,R_2$};
        \fill[pattern=horizontal lines, pattern color=gray, line width=0.5mm,xshift=32cm] (3,-6) circle (25pt);

        \end{tikzpicture}
        
        \caption{The regions $R_1$ and $R_2$ in the proof of Lemma \ref{lemma: zvar1}}
        \label{fig: lemmavar}
\end{figure}
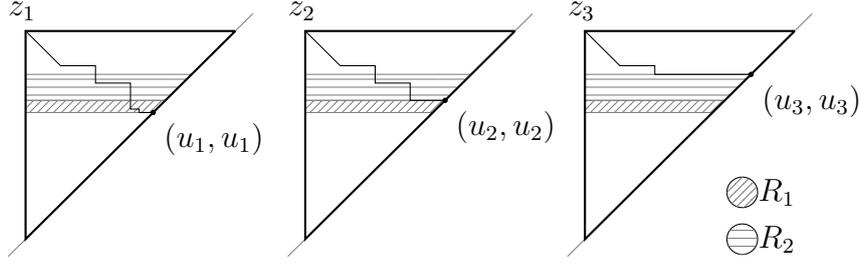

\begin{prop}\label{prop: vartot}
    Let $u$ and $z$ be solutions to the Riemann problem with initial data $u_0$ and $z_0$ we have that \begin{equation}
        Var(u(\cdot,t))=Var(u_0) \quad\text{and} \quad Var(z(\cdot,t))=Var(z_0)
    \end{equation}
    for every fixed $t\in [0,+\infty).$
\end{prop}
\begin{proof}
    First equality is trivial since for every fixed $t$, $u(\cdot,t)$ takes value from $u_l$ to $u_r$ in a monotone way, so $ Var(u(\cdot,t))=|u_l-u_r|$. \par Let us denote by $z^*$ the configuration in the half-plane $x>0$ after all the rarefaction waves. We can see $z^*$ as the result of the hysteresis process with an input passing from $u_r$ to $u_l$ in a monotone way, starting from an initial configuration equal to $z_r$. Then, by exploiting the previous Lemma \ref{lemma: zvar1}, is easy to see that \[Var(z(\cdot,t))=d(z_l,z^*)+d(z^*,z_r).\] Now we claim that \[d(z_l,z^*)+d(z^*,z_r)=d(z_l,z_r)=Var(z_0).\] Indeed, $z_r$ and $z^*$ may differ only in the strip $R_1=[-a,u_l]\times[u_r,u_l]\cap \T$, where certainly $z^*$ and $z_l$ are both equal to $+1$ since $R_1$ is below the input data $u_l$, see Figure \ref{fig: lemmavar2}. They may only differ on the strip $R_2=[-a,u_l]\times[u_l,a]$ on which however $z^*$ and $z_r$ agree. Instead $z_l$ and $z_r$ differ on $R_1\cup R_2$. See again Figure \ref{fig: lemmavar2}. Hence \[\begin{split} d(z_l,z_r)&=\int_{R_1} |z_l-z_r| ~d\p + \int_{R_2} |z_l-z_r| ~d\p \\&= \int_{R_1} |z^*-z_r| ~d\p + \int_{R_2} |z_l-z^*| ~d\p=d(z^*,z_l)+d(z_r,z^*).\end{split}\] 
\end{proof}
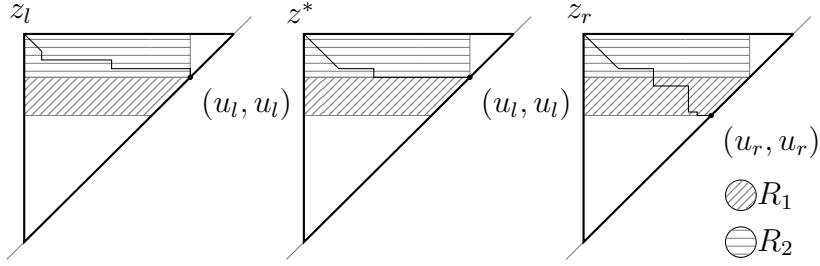
\begin{figure}
        \centering
        \begin{tikzpicture}[scale=0.23]

        \node[above] at (25.8,6) {$z_r$};
        \node[above] at (9.8,6) {$z^*$};
        \node[above] at (-6.2,6) {$z_l$};

        \draw[-,gray,xshift=32cm] (1.3,1.3)--(-6,1.3)--(-6,3.5)-- (3.5,3.5)--cycle;
        \fill[pattern = north east lines, pattern color = gray, line width = 0.5mm,xshift=32cm] (1.3,1.3)--(-6,1.3)--(-6,3.5)-- (3.5,3.5)--cycle;
        \draw[-,gray,xshift=32cm] (-6,3.5)-- (3.5,3.5)--(3.5,6)--(-6,6)--cycle;
        \fill[pattern = horizontal lines, pattern color = gray, line width = 0.5mm,xshift=32cm] (-6,3.5)-- (3.5,3.5)--(3.5,6)--(-6,6)--cycle;

        \draw[-,gray,xshift=32cm] (7,7)--(-7,-7);
        \draw[-,thick,xshift=32cm] (6,6)--(-6,-6)--(-6,6)--(6,6);
        \draw[-,xshift=32cm] (-6,6)--(-4,4)--(-2,4)--(-2,3)--(0,3)--(0,1.5)--(0.5,1.5)--(0.5,1.3)--(1.3,1.3) node[below right]{$(u_r,u_r)$};
        \filldraw[black,xshift=32cm] (1.3,1.3) circle (3.5pt);

        \draw[-,gray,xshift=16cm] (1.3,1.3)--(-6,1.3)--(-6,3.5)-- (3.5,3.5)--cycle;
        \fill[pattern = north east lines, pattern color = gray, line width = 0.5mm,xshift=16cm] (1.3,1.3)--(-6,1.3)--(-6,3.5)-- (3.5,3.5)--cycle;
        \draw[-,gray,xshift=16cm] (-6,3.5)-- (3.5,3.5)--(3.5,6)--(-6,6)--cycle;
        \fill[pattern = horizontal lines, pattern color = gray, line width = 0.5mm,xshift=16cm] (-6,3.5)-- (3.5,3.5)--(3.5,6)--(-6,6)--cycle;
        \draw[-,gray,xshift=16cm] (7,7)--(-7,-7);
        \draw[-,thick,xshift=16cm] (6,6)--(-6,-6)--(-6,6)--(6,6);
        
        \draw[-,xshift=16cm] (-6,6)--(-4,4)--(-2,4)--(-2,3.5)--(3.5,3.5) node[below right]{$(u_l,u_l)$};
        \filldraw[black,xshift=16cm] (3.5,3.5) circle (3.5pt);

       \draw[-,gray] (1.3,1.3)--(-6,1.3)--(-6,3.5)-- (3.5,3.5)--cycle;
        \fill[pattern = north east lines, pattern color = gray, line width = 0.5mm] (1.3,1.3)--(-6,1.3)--(-6,3.5)-- (3.5,3.5)--cycle;
        \draw[-,gray] (-6,3.5)-- (3.5,3.5)--(3.5,6)--(-6,6)--cycle;
        \fill[pattern = horizontal lines, pattern color = gray, line width = 0.5mm] (-6,3.5)-- (3.5,3.5)--(3.5,6)--(-6,6)--cycle;
        \draw[-,gray] (7,7)--(-7,-7);
        \draw[-,thick] (6,6)--(-6,-6)--(-6,6)--(6,6);
        
        \draw[-] (-6,6)--(-5,5)--(-5,4.5)--(-1,4.5)--(-1,4)--(3.5,4)--(3.5,3.5) node[below right]{$(u_l,u_l)$};
        \filldraw[black] (3.5,3.5) circle (3.5pt);

        \draw[black, xshift = 32cm] (3,-3.3) circle (25pt) node[right]{$\, R_1$};
        \fill[pattern=north east lines, pattern color=gray, line width=0.5mm, xshift = 32cm] (3,-3.3) circle (25pt);

        \draw[black, xshift = 32cm] (3,-6) circle (25pt) node[right]{$\,R_2$};
        \fill[pattern=horizontal lines, pattern color=gray, line width=0.5mm, xshift = 32cm] (3,-6) circle (25pt);

        \end{tikzpicture}
        \caption{The configuration $z_l,z^*,z_r$ respectively on the left, in the middle and on the right; we highlight the region $R_1$ on which $z_r$ and $z^*$ differ and the region $R_2$ on which $z^*$ and $z_l$ differ.}
        \label{fig: lemmavar2}
\end{figure}


\section{The general initial data Cauchy problem}\label{S5}

Exploiting the wave-front tracking method (\cite{AB3}, \cite{HH}), we construct a solution for the Cauchy problem \eqref{eq: hpde}. \par We start with initial data $u_0$ and $z_0$ piecewise constant, with a finite number of discontinuities. Taking inspiration from \cite[Chapter~6]{AB3}, we consider $u_0$ with values in the set $A_n:=\{ k\,2^{-n} \,|\, k\in \Z\}$, hence we denote the initial data as $u_0^{(n)}$. We know from previous Section \ref{S4} that if we solve the Riemann problem with $u_l,u_r\in A_n$ and general value $z_l, z_r$ we end up with a solution $u$ consisting of rarefaction waves connecting constant data $u_l, M_{k-1},\dots, u_r$, see \eqref{eq: uRP}. To ensure that also $M_i\in A_n$ we need to consider only initial configuration $z_0$ taking values in the set $\tilde{\mathcal S}_L^{(n)}$. We recall that such set, defined at the end of Section \ref{S2}, is the set of the configuration $z\in \tilde{\mathcal S}_L$ such that the corresponding graphs $\xi_z$ consist of a finite number of segments with vertices belonging to the set $Q_n= (A_n \times A_n) \cap \T.$ Considering $u_0(x)\in A_n$ and $z_{0}\in \tilde{\mathcal S}_L^{(n)}$ ensures that, after solving the Riemann problem, we still have $u(x,t)\in A_n$ and $z(x,t)\in \tilde{\mathcal S}_L^{(n)}$ for almost every $(x,t).$\par
Consider then the Cauchy problem with data $u^{(n)}_0: \R \to A_n, z^{(n)}_0: \R \to \tilde{\mathcal S}_L^{(n)}$ piecewise constant with a finite number of discontinuities and compatible, i.e. $(u_0^{(n)} (x), z_0^{(n)} (x ; \p))\in \LL_\p$ for almost every $x$ and $\p\in\T$. At time $t=0$ we solve a finite number of Riemann problems so we get a solution $u_n$, defined for small times, which is a finite collection of rarefaction waves. Since $z_0^{(n)}(x) \in \tilde{\mathcal S}_L^{(n)}$ then we have that the rarefactions connect data from the set $A_n$. Hence, at each rarefaction wave, we have left datum $j2^{-n}$ and right datum $(j+m)2^{-n}$ for some $j,m \in \Z$, so we can split the rarefaction in $|m|$ (non-entropical) shocks connecting $(j+i)2^{-n}$ and $(j+i+1)2^{-n}$ (we refer to \cite{garpic} for this kind of approach). In this way we also split the rarefaction waves of $z_n$ and $w_n$. The shock speed is then given by the generalized \RH conditions \eqref{eq: hrhcond}. As a result, these shocks are  discontinuities for both $u_n,z_n$ and hence $w_n$, whereas, in reference to Section \ref{S4}, the vertical ones are only for $z_n$ and $w_n$.

Then by the wave-front tracking procedure, we extend $u_n$ for larger times, and hence also $z_n$ and $w_n$. \par
We now analyze the possible wave interactions. Since at time $0$ we split the rarefaction waves by shocks of strength $2^{-n}$, then if two discontinuities of $u_n$ meet at some point $(x_1,t_1)$, with $t_1>0$ we have two possibilities: \begin{enumerate}
    \item one wave connects $j 2^{-n}$ and $(j+1) 2^{-n}$ and the other $(j+1) 2^{-n}$ and $j 2^{-n}$. So the new left and right data for $u_n$ are equal at $(x_1,t_1)$, hence we don't have discontinuities in $u_n$ after the interaction. We may only generate a new vertical discontinuity for $w_n$.
    \item one wave connects $j 2^{-n}$ and $(j+1) 2^{-n}$ and the other $(j+1) 2^{-n}$ and $(j+2) 2^{-n}$. However, such waves would not cross each other as the slope of the first one is greater due to the generalized \RH conditions \eqref{eq: hrhcond}. 
\end{enumerate}
We also notice that if we have more than two waves interacting, then, reasoning as in point $2$, the values separated by these waves must oscillate between $j 2^{-n}$ and $(j+1)2^{-n}$. Hence, after the interaction at most one shock is generated. In conclusion, since we start with a finite number of discontinuities for $u_n$ (each one of the finite quantity of rarefactions is split in a finite number of shocks), the number of discontinuities for $u_n$ stays finite, without increasing, implying a finite number of interactions. Finally, a wave of $u_n$ may interact with a vertical wave of $z_n$ at most once, and in that case no further waves of $u_n$ are generated, so we can conclude that also the number of this type of interactions is finite.\par 
We just proved the following statement. \begin{lem}\label{lemma: ndisc}
    The total number of discontinuities of $u_n$ is non-increasing and the number of interactions between discontinuity wave of $u_n$ and $w_n$ is finite. 
\end{lem}

So we can construct $u_n,z_n$ and also $w_n$ for each time $t>0$ having the following quite classical properties.

\begin{lem}\label{lemma: ph1} 
    Let $T>0$, $u_0^{(n)},z_0^{(n)}$ piecewise constant with a finite number of pieces, such that $(u_0^{(n)}(x),z_0^{(n)}(x;\p))\in\LL_\p$ for almost every $x$ and $\p\in \T$, $Var(u_0^{(n)})=C$ and $z_0^{(n)}(x)\in \tilde{\mathcal S}_L^{(n)}$ for some $L>0.$ Consider also $u_n: \R \times [0,T)\to \R $,   $z_n: \R \times [0,T)\to \tilde{\mathcal S}_{2\max(L,C/T)}^{(n)} $ and $w_n= \int_\T z_n \, d\p$ constructed via the wave-front tracking method starting with initial data $u_0^{(n)},z_0^{(n)}$. Then: \begin{enumerate}[i)]
    \item The couple $(u_n,w_n)$ satisfies the weak formulation of the PDE \eqref{eq: hweaksol} for any $T>0$, moreover $z_n(\p) =[h^\p(u_n,z_0^{(n)})]$ and $w_n=[\HM(u_n,z_0^{(n)})]$ hold, for almost every $x,t$ and $\p$.
    \item For fixed $t\in [0,+\infty)$ \begin{equation}\label{eq: lemmaph0}
        Var(u_n(\cdot,t))\leq Var (u_0^{(n)}(\cdot)), \quad Var(z_n(\cdot,t))\leq Var(z_0^{(n)}(\cdot)).
    \end{equation} 
    \item For $t,t'\in [0,+\infty)$ the following estimates hold \begin{equation}\label{eq: lemmaph1}
        \int_{-\infty}^{+\infty} |u_n(x,t)-u_n(x,t')|~dx \leq Var(u_0^{(n)}(\cdot))|t-t'|
    \end{equation} and \begin{equation}\label{eq: lemmaph2}
        \int_{-\infty}^{+\infty} d(z_n(x,t),z_n(x,t'))z~dx \leq Var(z^{(n)}_0(\cdot))|t-t'|.
    \end{equation}
\end{enumerate}
\end{lem}
\begin{proof}
    First we notice that $z(x,t)\in \tilde{\mathcal S}_{2\max(L,C/T)}^{(n)}$ is a consequence of Proposition \ref{prop: Ltilda}.
    $i)$ is true by construction, indeed any piecewise constant function satisfying the \RH condition is a weak solution. Also the hysteresis relationship holds by construction. $ii)$ is a direct consequence of Proposition \ref{prop: vartot}. In $iii)$ estimates \eqref{eq: lemmaph1} and \eqref{eq: lemmaph2} are quite standard, we omit the proof of them. (One can refer to \cite{AB3} or \cite{HH} for the proof in classical setting or to the paper \cite{BFMS} for a more detailed proof in the case with hysteresis).
\end{proof}

We now deal with the general initial data Cauchy problem \eqref{eq: hpde}. In the sequel, by $L^1(\mathbb{R};\tilde{\cal S}_L)$ (see Definition \ref{def: zL1BV}) we mean $\xi_v$ as ``zero" point in $\tilde{\cal S}_L$.

\begin{teo}\label{teo: esistenza}
    Let us fix $T>0$, and consider $u_0 \in BV(\R)\cap L^1(\R)$ and $z_0\in BV(\R;\tilde{\mathcal S}_L)\cap L^1(\R; \tilde{\mathcal S}_L)$, with $\xi_z(x)$ being a union of a finite number of segments for every $x$ and with $(u_0(x),z_0(x;\p))\in \LL_p$ for almost every $x$ and $\p \in \T$. Then there exists a couple $u,w\in C^0([0,T]; L^1_{loc}(\R))$ weak solution to the Cauchy problem \eqref{eq: hpde} (see Definition \ref{def: weak}). 
\end{teo}
\begin{proof}
    Combining the general theory of bounded variation functions (see \cite[Chapter~2]{AB3}), its adaptation to function with values in a metric space (see Lemmas \ref{lemma: BV1} and \ref{lemma: BV2}) and the approximation Lemma \ref{lemma: propPreis4} we can find the sequences $u_0^{(n)}$ and $z_0^{(n)} : \R \to \tilde{\mathcal S}_L$ of right-continuous, piecewise constant functions such that:   
     \begin{itemize}
        \item For almost every $x$, $u_0^{(n)}(x) \in A_n=\{k\,2^{-n}\,|\, k\in \mathbb Z\}$, $z_0^{(n)}\in \tilde{\mathcal S}_L^{(n)}\subset \tilde{\mathcal S}_L$ and they are compatible with each other; 
        \item $Var(u_0^{(n)})\leq Var(u_0):=C$ and $Var(z_0^{(n)})\leq Var(z_0)$;
        \item $ \lim_{n\to \infty} || u_0-u_n^{(n)}||_\infty =0$ and $\lim_{n\to \infty} ||d(z_0,z_0^{(n)})||_\infty = 0 $; 
    \end{itemize}
    Moreover, proceeding as proved in \cite[Chapter~3]{AB3}, since $u_0\in L^1(\R)$ and $z_0\in L^1(\R;\tilde{\mathcal S}_L)$, we can suppose also $u_0^{(n)}, z_0^{(n)}$ to converge respectively in $L^1(\R)$ to $u_0$ and in $L^1(\R;\tilde{\mathcal S}_L)$ to $z_0$.
    \par We can use $u_0^{(n)}$ and $z_0^{(n)}$ as initial data for our PDE and, by the wave-front tracking algorithm described above, construct the couple $u_n,z_n$ such that by setting \[w_n(x,t):=\int_{\T} z_n(x,t;\p) ~d\p\] then $(u_n,w_n)$ solves the PDE weakly. Notice that by Lemma \ref{lemma: propPreis1}, $z(x,t)\in \tilde{\mathcal S}_{L'}$ with $L'=2 \max\{L,C/T\}$ which is a compact subset of $\tilde{\mathcal S}$. Due to the regularity of the initial data $u_0$ and $z_0$, we can apply Theorem \ref{teo: BV3}, both in $\tilde S_{L'}$ and in $\mathbb{R}$, and Lemma \ref{lemma: ph1}, to get, by standard wave-front tracking limit procedure, the existence of two limits $u\in C^0([0,T];L^1(\R))$ and $z \in C^0([0,T)]; L^1(\R;\tilde{\mathcal S}_{L'}))$ such that, at least for subsequences, $u_n \to u$ in $L^1(\R\times (0,T))$, and $z_n\to z$ in $L^1(\R\times (0,T); \tilde{\mathcal S}_{L'})$. 
    By setting $w_n=\int_\T z_n \, d\p$ and $w=\int_\T z\, d\p$ we have \begin{equation}\begin{split}
        \int_0^T \int_\R |w-w_n| ~dxdt&= \int_0^T \int_\R \bigg|\int_\T z-z_n ~d\p \bigg| ~dxdt\\
        &\leq \int_0^T \int_\R 
        \int_\T |z-z_n|~d\p dxdt = \int_0^T \int_\R d(z,z_n)~dxdt
        \end{split}
    \end{equation} that is $w_n\to w$ in $L^1(\R \times (0,T))$. Moreover \begin{equation}
        \begin{split}
            \int_\R |w(x,t)-w(x,s)| ~ds &= \int_\R \bigg| \int_\T z(x,t;\p)-z(x,s;\p) ~d\p\bigg| ~ds\\ &\leq  \int_\R \int_\T |z(x,t;\p)-z(x,s;\p)| d\p~dx  \\ &= \int_\R d(z(x,t),z(x,s))~dx \leq C |t-s|,
        \end{split}
    \end{equation} that is $w\in C^0([0,t];L^1(\R))$.
    It is clear that $(u,w)$ satisfies the weak formulation \eqref{eq: hweaksol}. Moreover since \begin{equation}
        \int_0^T\int_{\mathbb{R}} \int_\T |z-z_n|~d\p dx dt\to 0
    \end{equation} then we have almost everywhere convergence for $(x,t,\p)\in[0,T]\times\mathbb{R}\times\T$ (up to a subsequence), so almost everywhere $|z|=1$ and \eqref{eq: hdis} holds.\par
    We only need to prove \eqref{eq: genweakhis}. Since for fixed $n$, $w_n = [\HM(u_n,z_0^{(n)})]$ for almost every $x$, we can also fix $t$ such that \eqref{eq: preisweak1} holds for the couple $(u_n(x,\cdot~),w_n(x,\cdot~))$ for almost every $x$. Integrating then this inequality on $\R$ we get \begin{equation}\label{eq: dis1}
        \int_{\T}\int_\R \int_{(0,t)} \tilde{u}_n d\left(\frac{\partial z_n}{\partial t} (x,\cdot~;\p)\right) ~dx d\p \geq \int_{\T}\int_\R \Psi_\p(z_n(x,\cdot~;\p);(0,t)) ~dxd\p ,
    \end{equation} where $\tilde{u}_n$ identifies the right continuous (w.r.t. to the time variable) representative of $u_n$ and $\left(\frac{\partial z_n}{\partial t} (x,\cdot~;\p)\right)$ the measure associated to the distribution $z'_n(x,\cdot~;\p)$, seen as a function of time with $x$ and $\p$ fixed. Now from a distributional point of view, for every test function $\phi$ with compact support in $\R\times (0,t)$ we have that (the last equality is the standard definition of a measure as integral of parametrized measures)
    \begin{equation}\label{eq: integraldistr1}
        \begin{split}
            \left\langle \frac{\partial w_n}{\partial t},\phi \right\rangle & = -\int_\R \int_{(0,t)} \frac{\partial \phi}{ \partial t} w_n ~dtdx \\ & =  -\int_\R \int_{(0,t)} \int_{\T} \frac{\partial \phi}{ \partial t} z_n ~d\p dtdx = -\int_{\T} \int_\R   \int_{(0,t)} \frac{\partial \phi}{ \partial t} z_n ~dtdxd\p \\ & = \int_{\T}  \int_\R \int_{(0,t)} \phi ~d\left(\frac{\partial z_n}{\partial t} (x, \cdot~;\p)\right) dx d\p = \left\langle \int_{\T} \frac{\partial z_n}{ \partial t} ~d\p, \phi \right\rangle,
        \end{split}
    \end{equation} but $w_n$ is also a $BV$ function from $\R \times (0,T)$ to $\R$ (since is a finite sum of constant pieces) so \begin{equation}\label{eq: integraldistr2}
        \left\langle \frac{\partial w_n}{\partial t},\phi \right\rangle = \int_\R \int_{(0,t)} \phi\, d\left(\frac{\partial w_n}{ \partial t} (x, \cdot) \right)~dx.
    \end{equation}
    Hence by \eqref{eq: integraldistr1}, \eqref{eq: integraldistr2} we get 
    $\partial w_n/\partial t=\int_{\T}(\partial z_n/\partial t)d\p$ as measures on $\R \times (0,t)$. Hence \eqref{eq: dis1} becomes \begin{equation}\label{eq: dis2}
       \int_\R \int_{(0,t)} \tilde{u}_n \,d\left(\frac{\partial w_n}{\partial t} (x,\cdot~)\right) ~dx  \geq \int_{\T}\int_\R \Psi_\p(z_n(x,\cdot~;\p);(0,t)) ~dxd\p.
    \end{equation} From the weak formulation \eqref{eq: hweaksol}, we get
    \begin{equation}\label{eq: ug2}
        \frac{\partial w_n}{\partial t}=-\frac{\partial u_n}{\partial t}-\frac{\partial u_n}{\partial x}
    \end{equation} in a distributional sense, which also holds in the sense of measures as $u_n$ is in $BV(\R \times (0,T))$. From \eqref{eq: dis2} and \eqref{eq: ug2} we get \begin{multline}\label{eq: dis3}
         -\int_\R \int_{(0,t)} \tilde{u}_n ~d\left(\frac{\partial u_n}{\partial t} (x,\cdot)\right) dx  -\int_{(0,t)}\int_\R  \tilde{u}_n ~d\left(\frac{\partial u}{\partial x}(\cdot,t)\right) dt\\ \geq \int_{\T} \int_\R \Psi_\p(z_n(x,\cdot~;d\p);(0,t)) ~d\p dx.
   \end{multline}
   Since for every fixed $x$, $\Tilde{u}_n(x,\cdot)_{|_{[0,t]}}=\sum_{i=1}^{N(x,t)} \Tilde{u}_n^{(i)}(x) \1 _{[t_{i-1},t_i)}$  then \begin{equation}\label{eq: misuraun}
   \int\limits_{(0,t)}\left(\Tilde{u}_n(x,t)\right)d \left(\frac{\partial u_n}{\partial t}(x,\cdot)\right) = \sum_{i=1}^{N(x,t)-1}\Tilde{u}_n^{(i+1)}(x) (\Tilde{u}_n^{(i+1)}(x)-\Tilde{u}_n^{(i)}(x)). \end{equation} Just neglecting the discontinuity points of $u_0^{(n)}$ and points $x$ such that $u_n$ is discontinuous in $(x,t)$, it is, for a.e. $x$, \[\Tilde{u}_n^{(1)}(x)=\lim_{s\to 0^+} u_n(x,s)=u_0^{(n)}(x), \quad \Tilde{u}_n^{(N(x,t))}(x)=\lim_{s\to t^-} u_n(x,s)=u_n(x,t).\] So by the standard inequality \[ \sum_{i=1}^{N-1} \Tilde{u}_n^{(i+1)}(x) (\Tilde{u}_n^{(i+1)}(x)-\Tilde{u}_n^{(i)}(x)) \geq \frac{1}{2} ((\Tilde{u}_n^{(N)})^2-(\Tilde{u}_n^{(1)})^2)\] we get from \eqref{eq: misuraun} \begin{equation}\label{eq: disu1}
       \int_\R \int_{(0,t)} \tilde{u}_n ~d\left(\frac{\partial u_n}{\partial t} (x,\cdot)\right) dx  \geq \frac{1}{2}\int_\R u_n(x,t)^2 -u_0^{(n)}(x)^2~dx.
   \end{equation}
   Similarly for fixed $t$, $\Tilde{u}_n(\cdot,t)=\sum_{i=1}^{N(t)} \Tilde{u}_n^{(i)}(t) \1 _{[x_{i-1},x_i]}$, which is left continuous, so by the same reasoning as above \begin{equation}\label{eq: disu2}
        \int_{(0,t)} \int_\R\tilde{u}_n ~d\left(\frac{\partial u_n}{\partial x} (\cdot,t)\right) dt  \geq \frac{1}{2}\int_{(0,t)} u_n(+\infty,t)^2 -u_n(-\infty,t)^2~ds=0,
   \end{equation} as $u_n(+\infty,t)=u_n(-\infty,t)=0$ since for almost $u_n(\cdot,t)\in L^1(\R)\cap BV(\R)$. Hence by \eqref{eq: dis3}, \eqref{eq: disu1} and \eqref{eq: disu2} we conclude that \begin{equation}\label{eq: dis4}
        \frac{1}{2}\int_\R \left(u_n(x,t)^2 -u_0^{(n)}(x)^2\right)~dx + \int_{\T} \int_\R \Psi_\p(z_n(x,\cdot~;d\p);(0,t)) ~d\p dx \leq 0,     
   \end{equation} which is \eqref{eq: genweakhis} for fixed $n$. In order to perform the limit as $n\to+\infty$, by the definition of $\Psi_\p$, we first write  \begin{equation}\begin{split}\label{eq: hweak2}
    \int_{\T} \int_\R& \Psi_\p(z_n(x,\cdot~;d\p);(0,t)) ~d\p dx= \\&= \int_{\T} \int_\R \left[\int_{(0,t)} \p_2 d \left(\frac{\partial z_n}{\partial t}^+ \right) - \int_{(0,t)} \p_1 d \left(\frac{\partial z_n}{\partial t}^-\right) \right]~ dx d\p \\ &=  \int_{\T} \int_\R \left[\int_{(0,t)} \frac{\p_2-\p_1}{2} d \left(\bigg|\frac{\partial z_n}{\partial t}\bigg| \right) + \int_{(0,t)} \frac{\p_2+\p_1}{2} d \left(\frac{\partial z_n}{\partial t}\right)\right] ~dx d\p \\ & =\int_{\T} \frac{\p_2-\p_1}{2}  \bigg|\frac{\partial z_n}{\partial t}\bigg|(\R \times (0,t);\p)~ d\p \\ & \quad+ \int_{\T} \int_{\R} \frac{\p_2+\p_1}{2} (z_n(x,t;\p)-z_0^{(n)}(x;\p) )~ dx d\p.
        \end{split}
   \end{equation} 
 The first term of the left-hand side of \eqref{eq: dis4} and the last term of the right-hand side of \eqref{eq: hweak2} respectively converge to 

   \[
\frac{1}{2}\int_\R \left(u(x,t)^2 -u_0(x)^2\right)\,dx,\ \  \int_{\T} \int_{\R} \frac{\p_2+\p_1}{2} (z(x,t;\p)-z_0(x;\p) )~ dx d\p
   \]
because of convergence in $L^1(\mathbb{R})$ and equiboundedness. Using \eqref{eq: dis4} and \eqref{eq: hweak2} and the above convergences, we can also infer \begin{equation}
       \limsup_{n\to \infty} \int_{\T} \frac{\p_2-\p_1}{2}  \bigg|\frac{\partial z_n}{\partial t}\bigg|(\R \times (0,t);\p)~ d\p  <\infty
   \end{equation} which is equivalent to saying that the mass of the sequence of measures $((\p_2-\p_1)/2) \partial z_n / \partial t:=((\p_2-\p_1)/2) \partial z_n / \partial t(\cdot;\rho)d\rho$ on the set $(\R \times (0,t) \times \T)$ is equibounded, as $(\p_2-\p_1)/2\geq 0$ on $\T$. So, up to a subsequence, there exists a measure weak star limit of $((\p_2-\p_1)/2)\partial z_n /\partial t$ which must then coincide with $((\p_2-\p_1)/ 2)\partial z / \partial t$  by the convergence of $z_n$ to $z$ in $L^1(\R\times(0,t)\times \T).$ Finally, by the lower-semicontinuity of the mass 
   
   \[
   \begin{split}
   \int_{\T} \frac{\p_2-\p_1}{2}  \bigg|\frac{\partial z}{\partial t}\bigg|(\R \times (0,t);\p)~ d\p &= \frac{\p_2-\p_1}{2}\bigg|\frac{\partial z}{\partial t} \bigg|(\R \times (0,t) \times \T)\\&\leq \liminf_{n\to \infty} \frac{\p_2-\p_1}{2}\bigg|\frac{\partial z_n}{\partial t}\bigg|(\R \times (0,t)\times \T)\end{split}
   \] 

   \noindent
   and, taking the liminf in \eqref{eq: dis4}, recalling once again \eqref{eq: hweak2}, we infer our desired condition \eqref{eq: genweakhis}.
\end{proof}

\section{Uniqueness}\label{S6}
We finally show that the solution constructed in Section \ref{S5} is the unique entropy solution to \eqref{eq: hpde}. Following \cite{AVH1}, by entropy solution here we mean a pair $(u,w)$ such that $w=\int_\T z \, d\p$ and the following entropy condition is satisfied \begin{equation}\label{eq: entropycond}
        \int_{\R}\int_0^T \left[|u-k|+\int_\T |z-\hat{z}|\, d\p\right] \phi_t+|u-k|\phi_x \,dt\,dx \geq0
    \end{equation} for every positive test function $\phi \in C^{\infty}( \R \times [0,+\infty))$ with compact support in $\R\times (0,+\infty)$, every $k\in \R$ and every function $\hat{z}:\T \to \{-1,+1\}, \hat{z}\in \tilde{\mathcal S}$ such that $(k,\hat{z}(\p))\in\LL_\p.$ \par 
Let us first check that the solution constructed for the Riemann problem is coherent with this definition.

\begin{prop}\label{prop: RPentropy}
    Consider $(u,w)$ with $w=\int_\T z \, d\p$ solution to the Riemann problem \eqref{eq: hpde}, \eqref{eq: datiR}, then it is an entropy solution. 
\end{prop}
\begin{proof}
    Suppose $u_l>u_r$. The case $u_l<u_r$ is similar. We have the following cases:\par
    \textbf{Case }$\boldsymbol{ u_r<u_l \leq k}$: For $x>0$, we have that $|u-k|$ is equal to $-u+k$ for almost every $(x,t)$ because $u$ consists of a family of rarefaction waves connecting values from $u_r$ to $u_l$ in a monotone way (see Figure \ref{fig: rarefazioni}). The configuration $z$ passes instead from the initial configuration $z_r$ to a configuration $z^*$, according to that monotone input $u$. By doing so we have a region $A\subset \T$ in the strip $([-a, u_l]\times [u_r,u_l]) \cap \T$ where, along the monotone evolution of $u$, $z$ changes from $-1$ to $+1$ (in Figure \ref{fig: lemmavar2}-right, $A$ is $R_1\cap\{\rho\ \mbox{``above the graph"}\}$). Since $(k,\hat{k}(\p))\in \LL_\p$ we know that in that strip $\hat{z}$ is certainly $+1$, hence in $A$, $|z-\hat{z}|$ changes from $2$ to $0$, and hence (recall $w=\int_\T z \,d\p$) 
    \begin{equation}\label{eq: -wt}
    \left(\int_\T |z-\hat{z}|\, d\p\right)_t = -w_t.
    \end{equation}
    For $x<0$ both $u$ and $z$ are constant. So by splitting the integral on the left hand side of \eqref{eq: entropycond} into the integral on the two regions $x<0, x>0$, by integrating by parts, by recalling that $u$ is continuous on $\{x=0\}$, and by \eqref{eq: -wt} we conclude that \eqref{eq: entropycond} holds as an equality.\par   
    \textbf{Case }$\boldsymbol{k\leq u_r<u_l}$: For $x<0$ both $u$ and $z$ are constant. For $x>0,$ $|u-k|=u-k$ and again $z$ changes from $-1$ to $+1$ in $A\subset \T$, hence also $|z-\hat{z}|$ changes only in $A$. However, now we only know that certainly $\hat{k} = -1$ in the region $\p_2 \geq k$ so $|z-\hat{z}|$ passes from $0$ to $2$ in the region $A\cap(\p_2 \geq k)$. In the rest of $A$ either it remains constant or it passes from $0$ to $2$ or from $2$ to $0$, in any case we deduce the bound \begin{equation}\label{eq: proofpropcaso2}
       \left(\int_\T |z-\hat{z}|\, d\p\right)_t \leq w_t.
   \end{equation} Then starting from the left hand side of \eqref{eq: entropycond}, splitting the integral over $x<0,x>0$, integrating by parts and using this last inequality \eqref{eq: proofpropcaso2}, we get \eqref{eq: entropycond}.\par
   \textbf{Case }$\boldsymbol{ u_r<k<u_l}$: For $x<0$, both $u,z$ are constant so we focus on $x>0,$ where again, we have that $z$ passes from $-1$ to $+1$ in $A\subset \T$. We need to subdivide the integral into two parts: one where $u>k$ and one where $u<k$. As we have a rarefaction wave type solution (see Figure \ref{fig: rarefazioni}-right) we have that the two regions are divided by the set $\{u=k\}$ which is either a half line or it lies between two half lines. Below $\{u=k\}$ we have $|u-k|=-u+k$ and the change of $z$ occurs only in $A \cap \{\p_2 \leq k\}$ on which $\hat{z}=-1$. This means again \[\left(\int_\T |z-\hat{z}|\,d\p\right)_t = -w_t.\] In the region above instead $|u-k|=u-k$ and \[\left(\int_\T |z-\hat{z}|\, d\p\right)_t\leq w_t.\] So by subdividing the integral into these parts, using the integration by parts formula and the last two equations we get \eqref{eq: entropycond}. Notice that both $|u-k|$ and $\int_\T |z-\hat{z}|\,d\p$ are continuous functions for $x>0$ hence the boundary terms in the integration by parts formula on $\{u=k\}$ cancel out.  
\end{proof}

\begin{rmk}
    Note that, even if the solution to our Riemann problem is just the union of rarefaction waves, without shocks, \eqref{eq: entropycond} may be strict, despite to the classical scalar case, where, for rarefaction waves only, the entropy condition always gives $0.$ 
\end{rmk}

We now prove the main theorem of this section.

\begin{teo}
    Given the initial conditions $u_0\in L^1(\R)\cap BV(\R),z_0\in L^1(\R;\tilde{\mathcal S}_L)\cap BV(\R;\tilde{\mathcal S}_L)$, let $(u,w)$ with $w=\int_\T z\, d\p$ be the solution to the PDE constructed via the wave-front tracking method. Then such couple is an entropy solution. 
\end{teo}
\begin{proof}
    Let us fix $n\in\N$, $k\in\R$, a configuration $\hat{k}$ as in the definition of entropy condition and a positive test function $\phi$. Let $(u_n,w_n)$ with $w_n=\int_\T z_n \, d\p$ be the solution to the approximating Cauchy problem with initial data $u^{(n)}_0,z_0^{(n)}$ constructed as in the proof of the previous Theorem \ref{teo: esistenza}. 
    We are going to prove that $(u_n,w_n)$ in an ``almost entropy condition" in the sense that \eqref{eq: entropycond} is satisfied up to an infinitesimal error in the right-hand side (see \eqref{eq: finaleentrop}). Passing to the limit as $n\to+\infty$ we will get the conclusion.
    
    Both $(u_n,w_n)$ and $z_n$ are constant on a finite number of triangles $T_j$. For each fixed time $t$ we have a family of points $x_i(t)$, $i=1,\dots, N(t)$ such that $x_i(t)< x_{i+1}(t)$ on which $u_n$ or $w_n$ have a jump point. Integrating by parts, the left side of \eqref{eq: entropycond} can be rewritten as
    \begin{equation}\label{eq: sommadisc}
        \begin{split}
        \sum_{i=1}^{N(t)}&\int_{0}^{T} \left[\left( \Delta u_i(t;k)+\Delta z_i(t;\hat{z}) \right) x_i'(t) - \Delta u_i(t;k)\right] \phi(x_i(t),t) \,dt
        \end{split} 
    \end{equation}
    where \[\Delta u_i(t;k):=|u(x_i(t)+,t)-k|-|u(x_i(t)-,t)-k|,\]
    \[\Delta z_i(t;\hat{z}):=\int_\T |z(x_i(t)+,t;\p)-\hat{z}(\p)|\,d\p-\int_\T |z(x_i(t)-,t;\p)-\hat{z}(\p)|\,d\p.\]
    We analyze all the possible different kind of discontinuities:  \par 
    
    $\boldsymbol{x'_i(t)=0:}$ vertical discontinuities which are only discontinuities for $z$ and $w$, hence, $\Delta u_i=0$, which means that they do not contribute in the sum in \eqref{eq: sommadisc}. \par
    $\boldsymbol{u(x_i(t)+,t)<u(x_i(t)-,t)\leq k:}$ in this case $\Delta u_i(t;k)=-u(x_i(t)+)+u(x_i(t)-)$ and by arguing as in the first case of the proof of Proposition \ref{prop: RPentropy} we can check that $\Delta z_i(t;\hat{z})= -w(x_i(t)+)+w(x_u(t)-)$. Now, since $x'(t)$ satisfies the \RH condition \eqref{eq: hrhcond} then also this contribution of these type of discontinuities in the sum in \eqref{eq: sommadisc} is equal to $0$.\par $\boldsymbol{k\leq u(x_i(t)+,t)<u(x_i(t)-,t)}$: here instead $\Delta u_i(t;k)=u(x_i(t)+)-u(x_i(t)-)$ and by doing the same reasoning as in the second case of the proof of Proposition \ref{prop: RPentropy} we can check that $\Delta z_i(t;\hat{z})\geq w(x_i(t)+)-w(x_u(t)-)$. Now, since $x'(t)$ satisfies the \RH condition \eqref{eq: hrhcond} it is easy to check that the contribution in \eqref{eq: sommadisc} of discontinuities belonging to this case  is positive.\par$\boldsymbol{ u(x_i(t)-,t)<u(x_i(t)+,t)\leq k}$ or $\boldsymbol{k\leq u(x_i(t)-,t)<u(x_i(t)+,t):}$ these cases are similar to the previous two ones and they at most give a positive  contribution to the sum in \eqref{eq: sommadisc} \par $\boldsymbol{u(x_i(t)+,t)<k<u(x_i(t)-,t)}$ or
    $\boldsymbol{u(x_i(t)-,t)<k<u(x_i(t)+,t):}$ discontinuities that belong to this case are the ones that contribute to the entropy error, indeed their contribution to the sum in \eqref{eq: sommadisc} may be negative. However we can estimate $\Delta u_i(t;k) \geq -2^{-n}$, $\Delta z_i(t;\hat{z}) \geq -C\, 2^{-n}$, $0\leq x'_i(t) \leq 1$ hence \begin{equation}
        \left( \Delta u_i(t;k)+\Delta z_i(t;\hat{z}) \right) x_i'(t) - \Delta u_i(t;k) \geq - \tilde{C}\, 2^{-n}
    \end{equation} for some $\tilde{C}\geq0$. \par
    Then by the previous analysis of the different kind of discontinuities, if we consider only the indices $i_j$ with $j=1,\dots, \tilde{N}(t)$ for which we might have negative contribution we can conclude from \eqref{eq: sommadisc} that, for some $\tilde{K}\geq 0$: \begin{equation}\begin{split}\label{eq: finaleentrop}
        \int_{0}^{+\infty} &\int_{-\infty}^{+\infty} \left[|u_n-k|+\int_\T |z_n-\hat{z}|\, d\p\right] \phi_t+|u_n-k|\phi_x \,dx\,dt \\ &\geq - \tilde{C}\, 2^{-n}\sum_{j=1}^{N(t)} \int_0^{+\infty } \phi(x_i(t),t)\,dt \geq - K \, 2^{-n} \, \tilde{N}(t) \geq -\tilde{K} \, n 2^{-n}
    \end{split}
    \end{equation} where the last inequality is due to the fact that  $``\boldsymbol{u(x_i(t)+,t)<k<u(x_i(t)-,t)}$ or
    $\boldsymbol{u(x_i(t)-,t)<k<u(x_i(t)+,t)}"$ can only occur in one discontinuity per rarefaction wave, so at time $0$ that number of occurrence is of the same order of $n$ (see the construction of $u_0^{(n)},z_0^{(n)}$ in the the proof of Theorem \ref{teo: esistenza}), and it may only decrease as time increases as stated in Lemma \ref{lemma: ndisc}. \par
    Taking the limit in \eqref{eq: finaleentrop} we conclude the proof. 
\end{proof}

As a consequence of the next theorem, we then deduce that the solution constructed via the wave-front tracking method is also the only entropy solution.

\begin{teo}
    Consider the Cauchy problems with initial conditions $u_0^1,z_0^1$ and $u_0^2,z_0^2$ respectively, where $u_0^i\in L^1(\R)\cap BV(\R)$, $z_0^i \in L^1(\R;\tilde{\mathcal S}_L)\cap BV(\R; \tilde{\mathcal S}_L) $. Let us denote by $(u_1,w_1), w_1=\int_\T z_1\, d\p$ and $(u_2,w_2), w_2=\int_\T z_2\,d\p$ two entropy solutions to the respective problems. Then it holds that \begin{equation}
        \int\limits_{-\infty}^{+\infty} (|u_1-u_2|(x,t)+\int\limits_\T|z_1-z_2|\,d\p\,(x,t))~dx \leq \int\limits_{-\infty}^{+\infty}(|u_0^1-u_0^2|+\int\limits_\T|z_0^1-z_0^2|\,d\p)~dx, 
    \end{equation} for almost every $t\in [0,T]$.
\end{teo}

\begin{proof}
    Adapting the standard method of doubling variables by Kruzkov \cite{Krukov}, in \cite{AVH1} an analogous result is proven for the case of delayed-Relay hysteresis. The adaptation to our case with Preisach hysteresis is not difficult. 
\end{proof}


\appendix
\section{BV functions with values in metric spaces}\label{A1}

In this appendix we state the main properties we used in the article about functions $f: \R \to E$ of bounded variation and with $E$ a metric space with some compactness properties. The proofs are omitted but they can be done in a similar way, as in the case $E=\R$, by exploiting compactness properties of $E.$ For these proofs we refer to \cite[Chapter~3]{AB3} and to \cite{AFP} for a general reference to metric space valued BV functions.\par 
Suppose we have a metric space $(E,d)$ and that it is complete. Consider a function $z: \R \to E$ and fix a point $z_v \in E$ as ``zero".

\begin{defi}\label{def: zL1BV}
    We say that $z\in L^1(\R;E)$ if \begin{equation}
        \int_{\R} d(z(x),z_v)~dx < +\infty.
    \end{equation}\par
    We say that $z\in BV(\R;E)$ if there exists a representative $v$ of it, i.e. $z=v$ almost everywhere, such that \begin{equation}\label{eq: bv}
    \sup\left\{ \sum_{i=0}^{N-1} d(v(x_i),v(x_{i+1})) \quad \bigg| \quad -\infty=x_0<x_1<\dots<x_N=\infty\right\} < +\infty.
    \end{equation}
\end{defi}
The total variation of $z$ is defined as the minimum value of \eqref{eq: bv} among all representatives of $z$. Such minimum is achieved and from now on when we say $z\in BV$ we will assume that $z$ is one of the representatives with the smallest total variation (see \cite{AFP}).

\begin{lem}\label{lemma: BV1}
    Let $z\in BV(\R,E)$ where $E$ is a complete metric space. Then for every $x\in \R$, the left and right limits of $z$ at $x$ exist.
\end{lem}

It can also be shown that right or left continuous representatives of $z$ minimize \eqref{eq: bv}, and thus they are called ``good" representatives and can be chosen in order to compute the total variation of $z$.

\begin{lem}\label{lemma: BV2}
    Let $z\in BV(\R,E)$ where $E$ is a complete metric space. Then for every $\varepsilon>0$, there exists a piecewise constant function $v$ such that \begin{equation}
        Var(v)\leq Var(z) \quad \text{and} \quad ||d(z,v)||_{L^\infty(\R)}\leq \varepsilon.
    \end{equation}
    If $z\in L^1(\R;E)$, then $v$ can be chosen in a way such that it also holds
    \begin{equation}
        ||d(z,w)||_{L^1(\R)}\leq \varepsilon.
    \end{equation}
\end{lem}
Regarding the compactness result, we need not only $E$ to be complete but also compact. 
\begin{teo}\label{teo: BV3}
    Consider a sequence of functions $z_n: \R \to E$, with $E$ compact metric space, such that \begin{equation}
        Var(z_n)\leq C, \quad ||d(z_n,z_v)||_{L^\infty(\R)}\leq M. 
    \end{equation}
    Then there exists a subsequence $z_{n_k}$ and $z\in BV(\R; E)$ such that \begin{equation} \label{eq: teohelly1}
        \lim_{k\to \infty} z_{n_k}(x)=z(x), \quad Var(z)\leq C, \quad ||d(z,z_v)||_{L^\infty(\R)} \leq M.
    \end{equation}
    Moreover if it also holds that \begin{equation}
    \lim_{n\to\infty}||d(z_n,z)||_{L^1_{loc}(\R)} =0.
    \end{equation}
\end{teo}

\bigskip

{\bf Acknowledgement.} The authors were partially supported by the GNAMPA-INdAM project 2025: Analisi e controllo di modelli evolutivi con fenomeni non locali.


\printbibliography

\end{document}